\documentclass[reqno,11pt]{amsart}
\usepackage{amssymb}
\usepackage{mathrsfs}
\usepackage{}
\usepackage{amsmath,amssymb}
\usepackage{paralist}
\usepackage{graphics} %% add this and next lines if pictures should be in esp format
\usepackage{epsfig} %For pictures: screened artwork should be set up with an 85 or 100 line screen
\usepackage[colorlinks=true]{hyperref}
\hypersetup{urlcolor=red, citecolor=blue}
\usepackage[top=1in,bottom=1in,left=1in,right=1in]{geometry}
\usepackage{cite}
\newtheorem{thm}{Theorem}[section]

\newtheorem{lem}[thm]{Lemma}

\theoremstyle{definition}
\newtheorem{defn}{Definition}[section]
\newtheorem{rem}{Remark}[section]
\numberwithin{equation}{section}
 \newcommand{\be}{\begin{equation}}
 \newcommand{\ee}{\end{equation}}
 \newcommand\bes{\begin{eqnarray}}
 \newcommand\ees{\end{eqnarray}}
 \newcommand{\bess}{\begin{eqnarray*}}
 \newcommand{\eess}{\end{eqnarray*}}
 
\begin{document}
\title[chemotaxis-Navier-Stokes system]
      {Global weak solutions for the three-dimensional chemotaxis-Navier-Stokes system with nonlinear diffusion}%
\author[Zhang]{Qingshan Zhang }%
\address{Department of Mathematics, Southeast University, Nanjing 211189, P. R. China}
\email{qingshan11@yeah.net}

%%%%P13 address changed (we have some new administrative rules)

\author[Li]{Yuxiang Li }%
\address{Department of Mathematics, Southeast University, Nanjing 211189, P. R. China}
\email{lieyx@seu.edu.cn}

\thanks{Supported in part by National Natural Science Foundation of China (No. 11171063).}


%\date{}%
%\dedicatory{}%
%\commby{}%
% ----------------------------------------------------------------

\begin{abstract}
We consider an initial-boundary value problem for the incompressible chemotaxis-Navier-Stokes equations generalizing the  porous-medium-type diffusion model
\begin{eqnarray*}
  \left\{\begin{array}{lll}
     \medskip
     n_t+u\cdot\nabla n=\Delta n^m-\nabla\cdot(n\chi(c)\nabla c),&{} x\in\Omega,\ t>0,\\
     \medskip
     c_t+u\cdot\nabla c=\Delta c-nf(c),&{}x\in\Omega,\ t>0,\\
     \medskip
     u_t+\kappa(u\cdot\nabla)u=\Delta u+\nabla P+n\nabla\Phi ,&{}x\in\Omega,\ t>0,\\
     \medskip
     \nabla\cdot u=0, &{}x\in\Omega,\ t>0,
  \end{array}\right.
\end{eqnarray*}
in a bounded convex domain $\Omega\subset\mathbb{R}^3$. It is proved that if $m\geq\frac{2}{3}$, $\kappa\in\mathbb{R}$, $0<\chi\in C^2([0,\infty))$, $0\leq f\in C^1([0,\infty))$ with $f(0)=0$ and $\Phi\in W^{1,\infty}(\Omega)$, then for sufficiently smooth initial data $(n_0, c_0, u_0)$ the model possesses at least one global weak solution.
\end{abstract}
\maketitle
% ----------------------------------------------------------------

\section{Introduction}
Chemotaxis is the directed movement of living cells under the effects of chemical gradients. Aerobic bacteria such as {\it{Bacillus subtilis}}
often live in thin fluid layers near solid-air-water contact line, in which the swimming bacteria move towards higher concentration of oxygen according to mechanism of chemotaxis and meanwhile the movement of fluid is under the influence of gravitational force generated by bacteria themselves. Both the oxygen concentration and bacteria density are transported by the fluid and diffuse through the fluid (\cite{Dombrowski-PRL-2004,Tao&Winkler-AIHP-2013,Lorz-M3AS-2010}).

To model such biological processes, Tuval et al. \cite{Tuval-P-2005} proposed the following model
\begin{eqnarray}\label{CNS1}
  \left\{\begin{array}{lll}
     \medskip
     n_t+u\cdot\nabla n=\Delta n-\nabla\cdot(n\chi(c)\nabla c),&{} x\in\Omega,\ t>0,\\
     \medskip
     c_t+u\cdot\nabla c=\Delta c-nf(c),&{}x\in\Omega,\ t>0,\\
     \medskip
     u_t+\kappa(u\cdot\nabla)u=\Delta u+\nabla P+n\nabla\Phi ,&{}x\in\Omega,\ t>0,\\
     \medskip
     \nabla\cdot u=0, &{}x\in\Omega,\ t>0
  \end{array}\right.
\end{eqnarray}
in a domain $\Omega\subset\mathbb{R}^N$, where the scalar functions $n=n(x,t)$ and $c=c(x,t)$ denote bacterial density and the concentration of oxygen, respectively. The vector $u=(u_1(x,t), u_2(x,t), \cdots, u_N(x,t))$ is the fluid velocity field and the associated pressure is represented by $P=P(x,t)$. The function $\chi$ is called the chemotactic sensitivity, $f$ is the consumption rate oxygen by the bacteria and $\kappa\in\mathbb{R}$ measures the strength of nonlinear fluid convection. The given function $\Phi$ stands for the gravitational potential produced by the action of physical forces on the cell.

The chemotaxis fluid system has been studied in the last few years and the main focus is on the solvability result. Under the assumption that $\chi(c)=\chi$ is a constant and $f$ is monotonically increasing with $f(0)=0$, Lorz \cite{Lorz-M3AS-2010} constructed local weak solutions in a bounded domain $\mathbb{R}^N$ $(N=2,3)$ with no-flux boundary condition and in $\mathbb{R}^2$ in the case of inhomogeneous Dirichlet conditions for oxygen. In bounded convex domains $\Omega\subset\mathbb{R}^2$, Winkler \cite{Winkler-CPDE-2012} proved that the initial-boundary value problem for (\ref{CNS1}) possesses a unique global classical solution. In \cite{Winkler-ARMA-2014} the same author showed that the global classical solutions obtained in \cite{Winkler-CPDE-2012} stabilize to the spatially uniform equilibrium $(\bar{n}_0, 0, 0)$ with $\bar{n}_0:=\frac{1}{|\Omega|}\int_{\Omega}n_0(x)dx$ as $t\rightarrow\infty$. Zhang and Li \cite{Zhang&Li-DCDSb-2015} proved that such solution converges to the equilibrium $(\bar{n}_0, 0, 0)$ exponentially in time. By deriving a new type of entropy-energy estimate, Jiang et al. \cite{Jiang&Zhengsongmu-2014} generalized the result of \cite{Winkler-ARMA-2014} to general bounded domains. For the well-posedness of the Cauchy problem to (\ref{CNS1}) in the whole space we refer the reader to  \cite{Duan&Lorz&Markowich-CPDE-2010,Liu&Lorz-AIHP-2011,Chae-DCDS-2013,ZhangQian-NARWA-2014,Zhang&Zheng-SIAM-2014,Chae-CPDE-2014}.

When the nonlinear convective term is ignored ($\kappa=0$ in (\ref{CNS1})), which means the fluid motion is slow, and the model is simplified to the chemotaxis-Stokes equation. In this modified version, global weak solutions are constructed for the two-dimensional Cauchy problem \cite{Duan&Lorz&Markowich-CPDE-2010}. In a bounded convex domain $\Omega\subset\mathbb{R}^3$, the chemotaxis-Stokes system possesses at least one global weak solution \cite{Winkler-CPDE-2012}.

The diffusion of bacteria may depend nonlinearly on their densities \cite{Hillen-JMB-2009,Vazquez-2007,Tao&Winkler-AIHP-2013,Tao&Winkler-DCDS-2012}. Introducing this into the model (\ref{CNS1}) leads to the chemotaxis-Navier-Stokes system with nonlinear diffusion\cite{Di&Lorz&Markowich-DCDS-2010}
\begin{eqnarray}\label{CNS}
  \left\{\begin{array}{lll}
     \medskip
     n_t+u\cdot\nabla n=\nabla\cdot\left(D(n)\nabla n\right)-\nabla\cdot(n\chi(c)\nabla c),&{} x\in\Omega,\ t>0,\\
     \medskip
     c_t+u\cdot\nabla c=\Delta c-nf(c),&{}x\in\Omega,\ t>0,\\
     \medskip
     u_t+\kappa(u\cdot\nabla)u=\Delta u+\nabla P+n\nabla\Phi ,&{}x\in\Omega,\ t>0,\\
     \medskip
     \nabla\cdot u=0, &{}x\in\Omega,\ t>0.
  \end{array}\right.
\end{eqnarray}
Up to now, the main issue of investigation to (\ref{CNS}) seems to concentrate on the chemotaxis-Stokes variant. Under the assumption $D(n)=mn^{m-1}$, Di Francesco et al. \cite{Di&Lorz&Markowich-DCDS-2010} proved that when $m\in(\frac{3}{2}, 2]$ the chemotaxis-Stokes system admits a global-in-time solution for general initial data in the bounded domain $\Omega\subset\mathbb{R}^2$. The same result holds in three-dimensional setting under the constraint $m\in(\frac{7+\sqrt{217}}{12}, 2]$. Tao and Winkler \cite{Tao&Winkler-DCDS-2012,Tao&Winkler-AIHP-2013} extended the global existence result so as to cover the whole range $m\in(1, \infty)$ in the bounded domain $\Omega\subset\mathbb{R}^2$ and $m\in(\frac{8}{7}, \infty)$ in the bounded convex domain $\Omega\subset\mathbb{R}^3$. In \cite{Liu&Lorz-AIHP-2011}, global existence of weak solution to the Cauchy problem of chemotaxis-Stokes system is established with $m=\frac{4}{3}$ in $\Omega=\mathbb{R}^2$. Recently, Duan and Xiang \cite{Duan&Xiang-IMRN-2014} generalized the global existence result for all exponents $m\in[1, \infty)$ .

In contrast to the chemotaxis-Stokes system, very few results of global solvability are available for the full nonlinear chemotaxis-Navier-Stokes system. In the case $\Omega\subseteqq \mathbb{R}^2$, global weak solutions are constructed by setting $D(n)=mn^{m-1}$ with $m\in[1, \infty)$ \cite{Duan&Xiang-IMRN-2014}. For the three-dimensional initial-boundary value problem, the only result we are aware of is that when $m>\frac{4}{3}$ the full system with nonlinear diffusion admits a global weak solution provided that $\Phi\in L^1_{loc}((0, \infty); L^1_{loc}(\Omega))$ with $\nabla\Phi\in L^2_{loc}((0, \infty); L^{\infty}(\Omega))$, and $\chi$ and $f$ are continuous differentiable satisfying $\chi'\geq0$, $f\geq0$ and $f(0)=0$ \cite{Vorotnikov-CMS-2014}.

Recently, for sufficiently smooth initial data $(n_0, c_0, u_0)$, Winkler \cite{winkler2014global} established global weak solutions of (\ref{CNS1}) in bounded convex domains $\Omega\subset\mathbb{R}^3$ under the assumptions $\chi\in C^2([0,\infty))$, $f\in C^1([0,\infty))$ with $f(0)=0$ and $\Phi\in W^{1,\infty}(\Omega)$. Motivated by the work of \cite{winkler2014global},
our purpose of the present paper is to consider the full chemotaxis-Navier-Stokes system with nonlinear diffusion. In order to formulate our result, we specify the precise mathematical setting: we shall subsequently consider (\ref{CNS}) along with boundary conditions
\begin{equation}\label{boundary c}
D(n)\frac{\partial n}{\partial\nu}=\frac{\partial c}{\partial\nu}=0,\quad u=0,\quad x\in\partial\Omega,\ t>0
\end{equation}
and the initial conditions
\begin{equation}\label{initial c}
n(x,0)=n_0(x),\quad c(x,0)=c_0(x),\quad u(x,0)=u_0(x),\quad x\in\Omega
\end{equation}
in a bounded convex domains $\Omega\subset\mathbb{R}^3$ with smooth boundary, where we assume
\begin{eqnarray}\label{initial data assumption}
  \left\{\begin{array}{lll}
     \medskip
     n_0\in L\log L(\Omega)\ \mbox{is positive},\\
     \medskip
     c_0\in L^{\infty}(\Omega)\ \mbox{is nonnegative and such that}\ \sqrt{c_0}\in W^{1,2}(\Omega), \\
     \medskip
     u_0\in L_{\sigma}^{2}(\Omega).
  \end{array}\right.
\end{eqnarray}
With respect to the parameter function in (\ref{CNS}), we shall suppose throughout the paper that
\begin{equation}\label{D1}
D(s)\in C_{loc}^{1+\gamma}\left((0, \infty)\right)\quad\mbox{for some}\ \gamma>0,
\end{equation}
\begin{equation}\label{D2}
D_1s^{m-1}\leq D(s)\leq D_2s^{m-1}\quad\mbox{for all}\ s>0
\end{equation}
with $m\geq\frac{2}{3}$ and $D_2\geq D_1>0$, and that
\begin{eqnarray}\label{D3}
  \left\{\begin{array}{lll}
     \medskip
     \chi\in C^2([0,\infty)),\quad \chi>0\quad \mbox{in}\ [0,\infty),\\
     \medskip
     f\in C^1([0,\infty)),\quad f(0)=0,\quad f>0\quad \mbox{in}\ (0,\infty),\\
     \medskip
     \Phi\in W^{1,\infty}(\Omega).
  \end{array}\right.
\end{eqnarray}
Moreover, we shall require the further technical assumptions
\begin{equation}\label{D4}
\left(\frac{f}{\chi}\right)'>0,\quad\mbox{on}\ [0,\infty)
\end{equation}
\begin{equation}\label{D5}
\left(\frac{f}{\chi}\right)''\leq0,\quad\mbox{on}\ [0,\infty)
\end{equation}
and
\begin{equation}\label{D6}
\left(\chi\cdot f\right)'\geq0,\quad\mbox{on}\ [0,\infty).
\end{equation}

Our main result reads as follows.
\begin{thm}\label{main result}
Let $\Omega\subset\mathbb{R}^3$ be a bounded convex domain with smooth boundary and $\kappa\in\mathbb{R}$. Suppose that the assumptions (\ref{initial data assumption})-(\ref{D6}) hold. Then there exists at least one global weak solution (in the sense of Definition \ref{global weak sol} below) of (\ref{CNS})-(\ref{initial c}) such that
\begin{equation*}
n^{\frac{m}{2}}\in L_{loc}^{2}([0, \infty); W^{1,2}(\Omega))\quad\mbox{and}\quad c^{\frac{1}{4}}\in L_{loc}^{4}([0, \infty); W^{1,4}(\Omega)).
\end{equation*}
\end{thm}

\begin{rem} (i) If the diffusion function $D(u)\equiv1$ in (\ref{CNS}), this is consistent with the result of \cite{winkler2014global}.

(ii) Theorem \ref{main result} shows that the model (\ref{CNS})-(\ref{initial c}) possesses global weak solution even when the diffusion effect is rather mild. However, we have to leave open here whether the lower bound of diffusion exponent $m=\frac{2}{3}$ is optimal to guarantee global weak solvability.
\end{rem}

The rest of this paper is organized as follows. In Section 2, we introduce a family of regularized problems and give some preliminary properties. Based on an energy-type inequality, a priori estimates are given in Section 3. Section 4 is devoted to showing the global existence of the regularized problems. In Section 5, we further establish some $\varepsilon$-independent estimates.   Finally, we give the proof of the main result in Section 6.

{\it Notations.} Throughout the paper, for any vectors $v\in \mathbb{R}^3$ and $w\in \mathbb{R}^3$, we denote by $v\otimes w$ the matrix $A_{3\times3}$ with $a_{ij}=v_iw_j$ for $i,j\in\{1,2,3\}$. We set $L\log L(\Omega)$ is the standard Orlicz space and $L_{\sigma}^2(\Omega):=\left\{\varphi\in L^2(\Omega)|\nabla\cdot\varphi=0\right\}$ denotes the Hilbert space of all solenoidal vector in $L^2(\Omega)$. As usual $\mathcal {P}$ denotes the Helmholtz projection in $L^2(\Omega)$. We write $W_{0, \sigma}^{1,2}(\Omega):=W_{0}^{1,2}(\Omega)\cap L_{\sigma}^2(\Omega)$ and $C_{0,\sigma}^{\infty}(\Omega):=C_{0}^{\infty}(\Omega)\cap L_{\sigma}^2(\Omega)$. We represent $A$ as the realization of Stokes operator $-\mathcal {P}\Delta$ in $L_{\sigma}^2(\Omega)$ with domain $D(A):=W^{2,2}(\Omega)\cap W_{0, \sigma}^{1,2}(\Omega)$. Also $n(\cdot,t)$, $c(\cdot,t)$ and $u(\cdot,t)$ will be denoted sometimes by $n(t)$, $c(t)$ and $u(t)$.
%%%%%%%%%%%%%%%%%%%%%%%%%%%%%%%%%%%%%%%%%%%%%%%%%%%%%%%%%%%%%%%%%%%%%%%%%%%%%%%%%%%%%%%%%%%%%%%%%%%%%%%%%%%%%%%%%%%%%

\section{Regularized problem}
Our intention is to construct a global weak solution as the limit of smooth solutions of appropriately regularized problems. According to the idea from \cite{winkler2014global} (see also \cite{Tao&Winkler-AIHP-2013}), let us first consider the approximate problems
\begin{eqnarray}\label{RegularizedCNS}
  \left\{\begin{array}{lll}
     \medskip
     n_{\varepsilon t}+u_{\varepsilon}\cdot\nabla n_{\varepsilon}=\nabla\cdot\left(D_{\varepsilon}(n_{\varepsilon})\nabla n_{\varepsilon}\right) -\nabla\cdot(n_{\varepsilon}F'_{\varepsilon}(n_{\varepsilon})\chi(c_{\varepsilon})\nabla c_{\varepsilon}),&{} x\in\Omega,\ t>0,\\
     \medskip
     c_{\varepsilon t}+u_{\varepsilon}\cdot\nabla c_{\varepsilon}=\Delta c_{\varepsilon}-F_{\varepsilon}(n_{\varepsilon})f(c_{\varepsilon}),&{}x\in\Omega,\ t>0,\\
     \medskip
     u_{\varepsilon t}+\kappa(Y_{\varepsilon}u_{\varepsilon}\cdot\nabla)u_{\varepsilon}=\Delta u_{\varepsilon}+\nabla P_{\varepsilon}+n_{\varepsilon}\nabla\Phi ,&{}x\in\Omega,\ t>0,\\
     \medskip
     \nabla\cdot u_{\varepsilon}=0, &{}x\in\Omega,\ t>0,\\
     \medskip
     \frac{\partial n_{\varepsilon}}{\partial\nu}=\frac{\partial c_{\varepsilon}}{\partial\nu}=0,\quad u_{\varepsilon}=0,&{} x\in\partial\Omega,\ t>0,\\
     \medskip
     n_{\varepsilon}(x,0)=n_{0\varepsilon}(x),\quad c_{\varepsilon}(x,0)=c_{0\varepsilon}(x),\quad u_{\varepsilon}(x,0)=u_{0\varepsilon}(x),&{} x\in\Omega
  \end{array}\right.
\end{eqnarray}
for $\varepsilon\in(0,1)$, where the approximate initial data $n_{0\varepsilon}\geq0$, $c_{0\varepsilon}\geq0$ and $u_{0\varepsilon}$ satisfy
\begin{eqnarray}\label{R initial1}
  \left\{\begin{array}{lll}
     \medskip
     n_{0\varepsilon}\in C_0^{\infty}(\Omega),\\
     \medskip
     \int_{\Omega}n_{0\varepsilon}=\int_{\Omega}n_{0},\\
     \medskip
     n_{0\varepsilon}\rightarrow n_0,\quad \varepsilon\rightarrow0\quad\mbox{in}\ L\log L(\Omega),
  \end{array}\right.
\end{eqnarray}
\begin{eqnarray}\label{R initial2}
  \left\{\begin{array}{lll}
     \medskip
     \sqrt{c_{0\varepsilon}}\in C_0^{\infty}(\Omega),\\
     \medskip
     \|c_{0\varepsilon}\|_{L^{\infty}(\Omega)}\leq \|c_{0}\|_{L^{\infty}(\Omega)},\\
     \medskip
    \sqrt{c_{0\varepsilon}}\rightarrow \sqrt{c_{0}},\quad \varepsilon\rightarrow0\quad\mbox{a.e. in}\ \Omega\ \mbox{and in}\ W^{1,2}(\Omega),
  \end{array}\right.
\end{eqnarray}
and
\begin{eqnarray}\label{R initial3}
  \left\{\begin{array}{lll}
     \medskip
     u_{0\varepsilon}\in C_{0,\sigma}^{\infty}(\Omega),\\
     \medskip
     \|u_{0\varepsilon}\|_{L^{2}(\Omega)}=\|u_{0}\|_{L^{2}(\Omega)},\\
     \medskip
     u_{0\varepsilon}\rightarrow u_0,\quad \varepsilon\rightarrow0\quad\mbox{in}\ L^2(\Omega).
  \end{array}\right.
\end{eqnarray}
The approximate functions in (\ref{RegularizedCNS}) can be chosen as
\begin{equation}\label{RD}
D_{\varepsilon}(s):=D(s+\varepsilon),\quad\mbox{for all}\ s\geq0,
\end{equation}
\begin{equation*}%\label{RF}
F_{\varepsilon}(s):=\frac{1}{\varepsilon}\ln(1+\varepsilon s),\quad\mbox{for all}\ s\geq0,
\end{equation*}
and the standard Yosida approximate $Y_{\varepsilon}$ (\cite{Sohr-2001}) is defined by
\begin{equation*}%\label{RU}
Y_{\varepsilon}v:=(1+\varepsilon A)^{-1}v,\quad\mbox{for all}\ v\in L^2_{\sigma}(\Omega).
\end{equation*}
It is easy to verify our choice of $F_{\varepsilon}$ above guarantees that for each $\varepsilon\in(0,1)$
\begin{equation}\label{RF1}
0\leq F'_{\varepsilon}(s)=\frac{1}{1+\varepsilon s}\leq1,\quad\mbox{for all}\ s\geq0,
\end{equation}
\begin{equation}\label{RF2}
sF'_{\varepsilon}(s)=\frac{s}{1+\varepsilon s}\leq\frac{1}{\varepsilon},\quad\mbox{for all}\ s\geq0,
\end{equation}
\begin{equation}\label{RF3}
0\leq F_{\varepsilon}(s)\leq s,\quad\mbox{for all}\ s\geq0,
\end{equation}
and
\begin{equation*}%\label{RF4}
F_{\varepsilon}(s)\rightarrow1,\quad F'_{\varepsilon}(s)\rightarrow1,\quad\ \mbox{as}\ \varepsilon\rightarrow0\quad\mbox{for all}\ s\geq0.
\end{equation*}

The first lemma concerns the local solvability of the approximate problems (\ref{RegularizedCNS}). The proof is based on well-established methods involving the Schauder fixed point theorem, the standard regularity theory of parabolic equation and the Stokes system (for details see \cite{Winkler-CPDE-2012,winkler2014global,Tao&Winkler-AIHP-2013}, for instance).
\begin{lem}\label{Lem Local sol.}
For any $\varepsilon\in(0,1)$, there exist a maximal existence time $T_{\max,\varepsilon}\in (0,\infty]$ and determined functions $n_{\varepsilon}>0$, $c_{\varepsilon}>0$ and $u_{\varepsilon}$
\begin{eqnarray*}
&&n_{\varepsilon}\in C^0\left(\bar{\Omega}\times[0, T_{\max,\varepsilon})\right)\cap C^{2,1}\left(\bar{\Omega}\times(0, T_{\max,\varepsilon})\right),\\
&&c_{\varepsilon}\in C^0\left(\bar{\Omega}\times[0, T_{\max,\varepsilon})\right)\cap C^{2,1}\left(\bar{\Omega}\times(0, T_{\max,\varepsilon})\right)\quad\mbox{and}\\
&&u_{\varepsilon}\in C^0\left(\bar{\Omega}\times[0, T_{\max,\varepsilon})\right)\cap C^{2,1}\left(\bar{\Omega}\times(0, T_{\max,\varepsilon})\right)
\end{eqnarray*}
such that $(n_{\varepsilon}, c_{\varepsilon}, u_{\varepsilon})$ is a classical solution of (\ref{RegularizedCNS}) in $\Omega\times(0, T_{\max,\varepsilon})$. Moreover, if $T_{\max,\varepsilon}<\infty$, then
\begin{equation*}
\|n_{\varepsilon}(\cdot,t)\|_{L^{\infty}(\Omega)}+\|c_{\varepsilon}(\cdot,t)\|_{W^{1,q}(\Omega)}
+\|u_{\varepsilon}(\cdot,t)\|_{D(A^{\alpha})}\rightarrow\infty,\quad t\rightarrow T_{\max,\varepsilon}
\end{equation*}
for all $q>3$ and $\alpha>\frac{3}{4}$.
\end{lem}

The following estimates of $n_{\varepsilon}$ and $c_{\varepsilon}$ are basic but important in the proof of our result.
\begin{lem}\label{Lem L1 of u c}
For each $\varepsilon\in(0,1)$, we have
\begin{equation}\label{est. u L1}
\int_{\Omega}n_{\varepsilon}(\cdot, t)=\int_{\Omega}n_0\quad\mbox{for all}\ t\in(0, T_{\max,\varepsilon})
\end{equation}
and
\begin{equation}\label{est. c L infite}
\|c_{\varepsilon}(\cdot, t)\|_{L^{\infty}(\Omega)}\leq\|c_0\|_{L^{\infty}(\Omega)}=:M\quad\mbox{in}\ \Omega\times(0, T_{\max,\varepsilon}).
\end{equation}
\end{lem}
\begin{proof}
Integrating the first equation in (\ref{RegularizedCNS}) and using (\ref{R initial1}), we obtain (\ref{est. u L1}). Since $f\geq0$ by our assumption (\ref{D3}) and $F_{\varepsilon}\geq0$ by (\ref{RF3}), an application of the maximum principle to the second equation in (\ref{RegularizedCNS}) gives (\ref{est. c L infite}).
\end{proof}

\section{An energy-type inequality}
In this section, we shall utilize an energy inequality associated with the first two equations in (\ref{RegularizedCNS}) to establish a priori estimates. The inequality is frequently used in the literature (see \cite{Duan&Lorz&Markowich-CPDE-2010,Tao&Winkler-AIHP-2013,winkler2014global,Winkler-ARMA-2014}, for example) and it also will play an important role in our proof.
\begin{lem}\label{Lem energy I}
Let (\ref{D1})-(\ref{D6}) hold. There exists $K\geq1$ such that for any $\varepsilon\in(0,1)$, the solution of (\ref{RegularizedCNS}) satisfies
\begin{eqnarray}\label{energy I}
&&\frac{d}{dt}\left\{\int_{\Omega}n_{\varepsilon}\ln n_{\varepsilon}+\frac{1}{2}\int_{\Omega}|\nabla\Psi(c_{\varepsilon})|^2\right\}
     +\frac{1}{K}\left\{\int_{\Omega}\frac{D_{\varepsilon}(n_{\varepsilon})}{n_{\varepsilon}}|\nabla n_{\varepsilon}|^2+\int_{\Omega}\frac{|D^2c_{\varepsilon}|^2}{c_{\varepsilon}}+\int_{\Omega}\frac{|\nabla c_{\varepsilon}|^4}{c_{\varepsilon}^3}\right\}\nonumber\\
&&\leq K\int_{\Omega}|\nabla u_{\varepsilon}|^2\quad\mbox{for all}\ t\in(0, T_{\max,\varepsilon}),
\end{eqnarray}
where $\Psi(s):=\int^s_1\frac{d\sigma}{\sqrt{g(\sigma)}}$ with $g(s):=\frac{f(s)}{\chi(s)}$.
\end{lem}
\begin{proof}
The proof is based on the first two equations in (\ref{RegularizedCNS}) and integration by parts and detailed computations can be found in \cite[Lemmas 3.1-3.4]{winkler2014global}.
\end{proof}

Based on Lemma \ref{Lem energy I}, we can modify the above energy-type inequality (\ref{energy I}) to contain all components of $n_{\varepsilon}$, $c_{\varepsilon}$ and $u_{\varepsilon}$.
\begin{lem}\label{Lem energy I1}
Let $\Psi$ be as given by Lemma \ref{Lem energy I} and suppose that (\ref{D1})-(\ref{D6}) hold. Then for any $\varepsilon\in(0,1)$, there exists $C>0$ such that
\begin{eqnarray}\label{energy I1}
&&\frac{d}{dt}\left\{\int_{\Omega}n_{\varepsilon}\ln n_{\varepsilon}+\frac{1}{2}\int_{\Omega}|\nabla\Psi(c_{\varepsilon})|^2
     +K\int_{\Omega}|u_{\varepsilon}|^2\right\}\nonumber\\
     &&\qquad+\frac{1}{2K}\left\{D_1\int_{\Omega}(n_{\varepsilon}+\varepsilon)^{m-2}|\nabla n_{\varepsilon}|^2+\int_{\Omega}\frac{|D^2c_{\varepsilon}|^2}{c_{\varepsilon}}+\int_{\Omega}\frac{|\nabla c_{\varepsilon}|^4}{c_{\varepsilon}^3}+\int_{\Omega}|\nabla u_{\varepsilon}|^2\right\}\nonumber\\
&&\leq C\quad\mbox{for all}\ t\in(0, T_{\max,\varepsilon}),
\end{eqnarray}
where $D_1$ and $K$ are constants provided by (\ref{D2}) and Lemma \ref{Lem energy I}, respectively.
\end{lem}
\begin{proof}
Multiplying both sides of the third equation in (\ref{RegularizedCNS}) by $u_{\varepsilon}$ and integrating by parts over $\Omega$, we have
\begin{equation*}
\frac{1}{2}\frac{d}{dt}\int_{\Omega}|u_{\varepsilon}|^2+\int_{\Omega}|\nabla u_{\varepsilon}|^2=-\kappa\int_{\Omega}(Y_{\varepsilon}u_{\varepsilon}\cdot\nabla)u_{\varepsilon}\cdot u_{\varepsilon}+\int_{\Omega}n_{\varepsilon}\nabla\Phi\cdot u_{\varepsilon}\quad\mbox{for all}\ t\in(0, T_{\max,\varepsilon}).
\end{equation*}
Since $\nabla\cdot u_{\varepsilon}=0$ implies $\nabla\cdot Y_{\varepsilon}u_{\varepsilon}=0$, we thereby obtain
\begin{equation*}
\frac{1}{2}\frac{d}{dt}\int_{\Omega}|u_{\varepsilon}|^2+\int_{\Omega}|\nabla u_{\varepsilon}|^2=\int_{\Omega}n_{\varepsilon}\nabla\Phi\cdot u_{\varepsilon}\quad\mbox{for all}\ t\in(0, T_{\max,\varepsilon}).
\end{equation*}
Substituting this into (\ref{energy I}), we get
\begin{eqnarray}\label{energy I2}
&&\frac{d}{dt}\left\{\int_{\Omega}n_{\varepsilon}\ln n_{\varepsilon}+\frac{1}{2}\int_{\Omega}|\nabla\Psi(c_{\varepsilon})|^2
     +K\int_{\Omega}|u_{\varepsilon}|^2\right\}\nonumber\\
     &&\qquad+\frac{1}{K}\left\{\int_{\Omega}\frac{D_{\varepsilon}(n_{\varepsilon})}{n_{\varepsilon}}|\nabla n_{\varepsilon}|^2+\int_{\Omega}\frac{|D^2c_{\varepsilon}|^2}{c_{\varepsilon}}+\int_{\Omega}\frac{|\nabla c_{\varepsilon}|^4}{c_{\varepsilon}^3}\right\}+K\int_{\Omega}|\nabla u_{\varepsilon}|^2\nonumber\\
&&\leq2K\int_{\Omega}n_{\varepsilon}\nabla\Phi\cdot u_{\varepsilon}\quad\mbox{for all}\ t\in(0, T_{\max,\varepsilon}).
\end{eqnarray}
Using (\ref{D2}) and (\ref{RD}), we have for each $\varepsilon\in(0,1)$
\begin{equation}\label{energy I21}
\int_{\Omega}\frac{D_{\varepsilon}(n_{\varepsilon})}{n_{\varepsilon}}|\nabla n_{\varepsilon}|^2\geq D_1\int_{\Omega}(n_{\varepsilon}+\varepsilon)^{m-2}|\nabla n_{\varepsilon}|^2\quad\mbox{for all}\ t\in(0, T_{\max,\varepsilon}).
\end{equation}
By (\ref{D3}), H\"{o}der's inequality and the embedding $W^{1,2}(\Omega)\hookrightarrow L^{6}(\Omega)$ for $n=3$, we can find $C_1>0$ such that for each $\varepsilon\in(0,1)$
\begin{eqnarray*}
2K\int_{\Omega}n_{\varepsilon}\nabla\Phi\cdot u_{\varepsilon}&\leq&2K\|\Phi\|_{W^{1,\infty}}\|n_{\varepsilon}+\varepsilon\|_{L^{\frac{6}{5}}(\Omega)}\|u_{\varepsilon}\|_{L^{6}(\Omega)}\\
               &\leq&C_1\|n_{\varepsilon}+\varepsilon\|_{L^{\frac{6}{5}}(\Omega)}\|\nabla u_{\varepsilon}\|_{L^{2}(\Omega)}
                     \quad\mbox{for all}\ t\in(0, T_{\max,\varepsilon}).
\end{eqnarray*}
Note that
\begin{equation}\label{n0 L1}
\|(n_{\varepsilon}+\varepsilon)^{\frac{m}{2}}\|_{L^{\frac{2}{m}}(\Omega)}\leq\left(\|n_0\|_{L^{1}(\Omega)}+|\Omega|\right)^{\frac{m}{2}}\quad\mbox{for all}\ t\in(0, T_{\max,\varepsilon})
\end{equation}
by (\ref{R initial1}) and (\ref{est. u L1}). It follows from the Gagliardo-Nirenberg inequality \cite{Winkler-MMAS-2002} that
\begin{eqnarray*}
\left\|n_{\varepsilon}+\varepsilon\right\|_{L^{\frac{6}{5}}(\Omega)}&=&\left\|(n_{\varepsilon}+\varepsilon)^{\frac{m}{2}}\right\|
^{\frac{2}{m}}_{L^{\frac{12}{5m}}(\Omega)}\\
               &\leq&C_2\left(\left\|\nabla (n_{\varepsilon}+\varepsilon)^{\frac{m}{2}}\right\|^{\frac{1}{3m-1}}_{L^{2}(\Omega)}\cdot
                     \left\|(n_{\varepsilon}+\varepsilon)^{\frac{m}{2}}\right\|^{\frac{5m-2}{m(3m-1)}}_{L^{\frac{2}{m}}(\Omega)}
                     +\left\|(n_{\varepsilon}+\varepsilon)^{\frac{m}{2}}\right\|^{\frac{2}{m}}_{L^{\frac{2}{m}}(\Omega)}\right)\\
               &\leq&C_3\left(\left(\int_{\Omega}(n_{\varepsilon}+\varepsilon)^{m-2}|\nabla n_{\varepsilon}|^2\right)^{\frac{1}{2(3m-1)}}+1\right)
                     \quad\mbox{for all}\ t\in(0, T_{\max,\varepsilon})
\end{eqnarray*}
with $C_2>0$ and $C_3>0$. Since $\frac{1}{3m-1}\leq1$ by our assumption $m\geq\frac{2}{3}$ , Young's inequality yields $C_4>0$ such that
\begin{eqnarray}\label{energy I22}
2K\int_{\Omega}n_{\varepsilon}\nabla\Phi\cdot u_{\varepsilon}
          &\leq&C_1C_3\left(\left(\int_{\Omega}(n_{\varepsilon}+\varepsilon)^{m-2}|\nabla n_{\varepsilon}|^2\right)^{\frac{1}{2(3m-1)}}+1\right)
                \cdot\|\nabla u_{\varepsilon}\|_{L^{2}(\Omega)}\nonumber \\
          &\leq&\frac{1}{K}C_1C_3\left(\left(\int_{\Omega}(n_{\varepsilon}+\varepsilon)^{m-2}|\nabla n_{\varepsilon}|^2\right)^{\frac{1}{3m-1}}+1\right)
                +\frac{K}{2}\int_{\Omega}|\nabla u_{\varepsilon}|^2\nonumber \\
          &\leq&\frac{D_1}{2K}\int_{\Omega}(n_{\varepsilon}+\varepsilon)^{m-2}|\nabla n_{\varepsilon}|^2+\frac{K}{2}\int_{\Omega}|\nabla u_{\varepsilon}|^2
                +C_4
\end{eqnarray}
for all $t\in(0, T_{\max,\varepsilon})$. Inequality (\ref{energy I1}) then follows by combining (\ref{energy I2}), (\ref{energy I21}) and (\ref{energy I22}).
\end{proof}

We can now use Lemma \ref{Lem energy I1} to establish a priori estimates of the solution of (\ref{RegularizedCNS}).
\begin{lem}\label{Lem energy I2}
Let $\Psi$ and $K$ be as given by Lemma \ref{Lem energy I}, and assume that the requirements of Lemma \ref{Lem energy I1} are satisfied. Then there exists $C\geq0$ such that for any $\varepsilon\in(0,1)$ we have
\begin{eqnarray}\label{energy I3}
\int_{\Omega}n_{\varepsilon}\ln n_{\varepsilon}+\frac{1}{2}\int_{\Omega}|\nabla\Psi(c_{\varepsilon})|^2+K\int_{\Omega}|u_{\varepsilon}|^2
\leq C\quad\mbox{for all}\ t\in(0, T_{\max,\varepsilon})
\end{eqnarray}
and
\begin{eqnarray}\label{energy I4}
D_1\int_0^T\int_{\Omega}(n_{\varepsilon}+\varepsilon)^{m-2}|\nabla n_{\varepsilon}|^2+\int_0^T\int_{\Omega}\frac{|D^2c_{\varepsilon}|^2}{c_{\varepsilon}}+\int_0^T\int_{\Omega}\frac{|\nabla c_{\varepsilon}|^4}{c_{\varepsilon}^3}+\int_0^T\int_{\Omega}|\nabla u_{\varepsilon}|^2\leq C(T+1)
\end{eqnarray}
for all $T\in(0, T_{\max,\varepsilon})$.
\end{lem}
\begin{proof}
Set
\begin{equation}\label{y definition}
y_{\varepsilon}(t):=\int_{\Omega}n_{\varepsilon}\ln n_{\varepsilon}+\frac{1}{2}\int_{\Omega}|\nabla\Psi(c_{\varepsilon})|^2
     +K\int_{\Omega}|u_{\varepsilon}|^2\quad\mbox{for all}\ t\in(0, T_{\max,\varepsilon})
\end{equation}
and
\begin{equation*}
h_{\varepsilon}(t):=D_1\int_{\Omega}(n_{\varepsilon}+\varepsilon)^{m-2}|\nabla n_{\varepsilon}|^2+\int_{\Omega}\frac{|D^2c_{\varepsilon}|^2}{c_{\varepsilon}}+\int_{\Omega}\frac{|\nabla c_{\varepsilon}|^4}{c_{\varepsilon}^3}+\int_{\Omega}|\nabla u_{\varepsilon}|^2\quad\mbox{for all}\ t\in(0, T_{\max,\varepsilon}).
\end{equation*}
Then (\ref{energy I1}) implies
\begin{equation}\label{energy I5}
y'_{\varepsilon}(t)+\frac{1}{2K}h_{\varepsilon}(t)\leq C\quad\mbox{for all}\ t\in(0, T_{\max,\varepsilon}).
\end{equation}
In order to introduce dissipative term in (\ref{energy I5}), we show that $y_{\varepsilon}(t)$ is dominated by $h_{\varepsilon}(t)$. Now using the inequality
\begin{equation*}
z\ln z\leq\frac{3}{3m-1}z^{m+\frac{2}{3}}\quad\mbox{for all}\ z>0
\end{equation*}
with $m\geq\frac{2}{3}$, we can find positive constants $C_1$, $C_2$ and $C_3$ fulfilling for each $\varepsilon\in(0,1)$
\begin{eqnarray}\label{energy I51}
\int_{\Omega}n_{\varepsilon}\ln n_{\varepsilon}&\leq&\frac{3}{3m-1}\int_{\Omega}(n_{\varepsilon}+\varepsilon)^{m+\frac{2}{3}}\nonumber\\
           &=&\frac{3}{3m-1}\left\|(n_{\varepsilon}+\varepsilon)^{\frac{m}{2}}\right\|^{\frac{2(3m+2)}{3m}}_{L^{\frac{2(3m+2)}{3m}}(\Omega)}\nonumber\\
           &\leq&\frac{3C_1}{3m-1}\left(\left\|\nabla (n_{\varepsilon}+\varepsilon)^{\frac{m}{2}}\right\|^{\frac{3m}{3m+2}}_{L^{2}(\Omega)}\cdot
                     \left\|(n_{\varepsilon}+\varepsilon)^{\frac{m}{2}}\right\|^{\frac{2}{3m+2}}_{L^{\frac{2}{m}}(\Omega)}
                     +\left\|(n_{\varepsilon}+\varepsilon)^{\frac{m}{2}}\right\|_{L^{\frac{2}{m}}(\Omega)}\right)^{\frac{2(3m+2)}{3m}}\nonumber\\
               &\leq&C_2\left(\left\|\nabla (n_{\varepsilon}+\varepsilon)^{\frac{m}{2}}\right\|^{2}_{L^{2}(\Omega)}\cdot
                     \left\|(n_{\varepsilon}+\varepsilon)^{\frac{m}{2}}\right\|^{\frac{4}{3m}}_{L^{\frac{2}{m}}(\Omega)}
                     +\left\|(n_{\varepsilon}+\varepsilon)^{\frac{m}{2}}\right\|^{\frac{2(3m+2)}{3m}}_{L^{\frac{2}{m}}(\Omega)}\right)\nonumber\\
               &\leq&C_3\left(\int_{\Omega}(n_{\varepsilon}+\varepsilon)^{m-2}|\nabla n_{\varepsilon}|^2+1\right)
                     \quad\mbox{for all}\ t\in(0, T_{\max,\varepsilon})
\end{eqnarray}
by the Gagliardo-Nirenberg inequality and (\ref{n0 L1}). According to (\ref{D3}), we have
\begin{equation*}
g(s):=\frac{f(s)}{\chi(s)}\in C^1([0,M])\quad\mbox{and}\quad g(0)=0.
\end{equation*}
Hence the mean value theorem yields $g(s)\geq \min_{\tau\in[0,M]}g'(\tau)s=:\bar{M}s$. We now apply Young's inequality and (\ref{est. c L infite}) to obtain
\begin{eqnarray}\label{energy I52}
\frac{1}{2}\int_{\Omega}|\nabla\Psi(c_{\varepsilon})|^2&=&\frac{1}{2}\int_{\Omega}\frac{|\nabla
                     c_{\varepsilon}|^2}{g(c_{\varepsilon})}\nonumber\\
             &\leq&\frac{1}{4}\int_{\Omega}\frac{|\nabla
                     c_{\varepsilon}|^4}{c_{\varepsilon}^3}+\frac{1}{4}\int_{\Omega}\frac{c_{\varepsilon}^3}{g^2(c_{\varepsilon})}\nonumber\\
             &\leq&\frac{1}{4}\int_{\Omega}\frac{|\nabla
                     c_{\varepsilon}|^4}{c_{\varepsilon}^3}+\frac{M|\Omega|}{4\bar{M}^{2}}
                     \quad\mbox{for all}\ t\in(0, T_{\max,\varepsilon}).
\end{eqnarray}
From the Poincar\'e inequality, we have $C_4>0$ such that
\begin{equation}\label{energy I53}
K\int_{\Omega}|u_{\varepsilon}|^2\leq C_4\int_{\Omega}|\nabla u_{\varepsilon}|^2\quad\mbox{for all}\ t\in(0, T_{\max,\varepsilon}).
\end{equation}
It follows easily from (\ref{energy I51})-(\ref{energy I53}) that
\begin{equation*}
y_{\varepsilon}(t)\leq C_5h_{\varepsilon}(t)+C_6\quad\mbox{for all}\ t\in(0, T_{\max,\varepsilon})
\end{equation*}
with $C_5:=\max\left\{\frac{D_1}{C_3}, \frac{1}{4}, C_4\right\}$ and $C_6:=C_3+\frac{M|\Omega|}{4\bar{M}^{2}}$. This, along with (\ref{energy I5}), yields
\begin{equation*}
y'_{\varepsilon}(t)+\frac{1}{4KC_5}y_{\varepsilon}(t)+\frac{1}{4K}h_{\varepsilon}(t)\leq C_7:=C+\frac{C_6}{4KC_5}\quad\mbox{for all}\ t\in(0, T_{\max,\varepsilon}).
\end{equation*}
Noting that $h_{\varepsilon}(t)\geq0$, a standard ODE comparison argument implies
\begin{equation}\label{energy I54}
y_{\varepsilon}(t)\leq \max\left\{\sup_{\varepsilon\in(0, 1)}y_{\varepsilon}(0),\quad 4KC_5C_7\right\}\quad\mbox{for all}\ t\in(0, T_{\max,\varepsilon}).
\end{equation}
In view of (\ref{R initial1})-(\ref{R initial3}) and \cite[Lemma 3.7]{winkler2014global}, we obtain (\ref{energy I3}). On the other hand, since $z\ln z\geq-\frac{1}{e}$ for all $z>0$, we have $y_{\varepsilon}(t)\geq-\frac{|\Omega|}{e}$. Therefore, a time integration of (\ref{energy I54}) directly leads to (\ref{energy I4}).
\end{proof}

\section{Global existence for the regularized problem (\ref{RegularizedCNS})}
With Lemma \ref{Lem energy I2} at hand, we are now in the position to show the solution of approximate problem (\ref{RegularizedCNS}) is actually global in time.
\begin{lem}\label{Lem global exi.}
For each $\varepsilon\in(0,1)$, the solutions of (\ref{RegularizedCNS}) are global in time.
\end{lem}
\begin{proof}
In this section, we shall denote by $C$ various positive constants which may vary from step to step and which possibly depend on $\varepsilon$. Assume for contradiction that $T_{\max,\varepsilon}<\infty$ for some $\varepsilon\in(0,1)$. By Lemma \ref{Lem energy I2}, we know that
\begin{eqnarray}\label{global exi.1}
\int_{\Omega}|u_{\varepsilon}|^2\leq C\quad\mbox{for all}\ t\in(0, T_{\max,\varepsilon})
\end{eqnarray}
and
\begin{eqnarray}\label{global exi.2}
\int_0^{T_{\max,\varepsilon}}\int_{\Omega}|\nabla c_{\varepsilon}|^4&=&\int_0^{T_{\max,\varepsilon}}\int_{\Omega}\frac{|\nabla c_{\varepsilon} |^4}{c^3_{\varepsilon}}c^3_{\varepsilon}\nonumber\\
&\leq&M^3\int_0^{T_{\max,\varepsilon}}\int_{\Omega}\frac{|\nabla c_{\varepsilon}|^4}{c^3_{\varepsilon}}\nonumber\\
&\leq& C\quad\mbox{for all}\ t\in(0, T_{\max,\varepsilon}).
\end{eqnarray}
Multiplying the first equation in (\ref{RegularizedCNS}) by $p(n_{\varepsilon}+\varepsilon)^{p-1}$ with $p\in[m+1, 2(m+1)]$ and using integration by parts we obtain
\begin{eqnarray}\label{global exi.3}
&&\frac{d}{dt}\int_{\Omega}(n_{\varepsilon}+\varepsilon)^p+p(p-1)\int_{\Omega}(n_{\varepsilon}+\varepsilon)^{p-2}D_{\varepsilon}(n_{\varepsilon})|\nabla n_{\varepsilon}|^2\nonumber\\
&&\quad=p(p-1)\int_{\Omega}(n_{\varepsilon}+\varepsilon)^{p-2}n_{\varepsilon}F'_{\varepsilon}(n_{\varepsilon})\chi(c_{\varepsilon})\nabla c_{\varepsilon}\cdot \nabla n_{\varepsilon}
\end{eqnarray}
for all $t\in(0, T_{\max,\varepsilon})$. We deduce from (\ref{D2}), (\ref{RF2}) and Young's inequality that
\begin{eqnarray*}
&&\frac{d}{dt}\int_{\Omega}(n_{\varepsilon}+\varepsilon)^p+p(p-1)D_1\int_{\Omega}(n_{\varepsilon}+\varepsilon)^{m+p-3}|\nabla n_{\varepsilon}|^2\\
&&\quad\leq\frac{p(p-1)}{\varepsilon}\max_{s\in[0,M]}\chi(s)\int_{\Omega}(n_{\varepsilon}+\varepsilon)^{p-2}\nabla c_{\varepsilon}\cdot
        \nabla n_{\varepsilon}\\
&&\quad\leq p(p-1)D_1\int_{\Omega}(n_{\varepsilon}+\varepsilon)^{m+p-3}|\nabla n_{\varepsilon}|^2+\int_{\Omega}(n_{\varepsilon}+\varepsilon)^{2(p-m-1)}
        +C\int_{\Omega}|\nabla c_{\varepsilon}|^{4}
\end{eqnarray*}
for all $t\in(0, T_{\max,\varepsilon})$. Since $2(p-m-1)\leq p$ for $p\leq2(m+1)$, applying Young's inequality again, we obtain
\begin{eqnarray*}
\frac{d}{dt}\int_{\Omega}(n_{\varepsilon}+\varepsilon)^p\leq\int_{\Omega}(n_{\varepsilon}+\varepsilon)^p+C\int_{\Omega}|\nabla c_{\varepsilon}|^{4}+C
\quad\mbox{for all}\ t\in(0, T_{\max,\varepsilon}).
\end{eqnarray*}
Integrating this yields
\begin{equation}\label{global exi.p}
\int_{\Omega}(n_{\varepsilon}+\varepsilon)^p\leq C\quad\mbox{for all}\ t\in(0, T_{\max,\varepsilon}),
\end{equation}
where $p\in[1, 2(m+1)]$. We now use the idea from \cite{winkler2014global} to obtain the boundedness of $u_{\varepsilon}$. From (\ref{global exi.1}), we get
\begin{eqnarray}\label{global exi.4}
\|Y_{\varepsilon}u_{\varepsilon}(t)\|_{L^{\infty}(\Omega)}&=&\|(1+\varepsilon A)^{-1}u_{\varepsilon}(t)\|_{L^{\infty}(\Omega)}\nonumber\\
            &\leq&C\|u_{\varepsilon}(t)\|_{L^{2}(\Omega)}\nonumber\\
            &\leq&C\quad\mbox{for all}\ t\in(0, T_{\max,\varepsilon})
\end{eqnarray}
due to the embedding $D(1+\varepsilon A)\hookrightarrow L^{\infty}(\Omega)$. We apply the Helmholtz projection $\mathcal {P}$ to the third equation in (\ref{RegularizedCNS}), test the resulting identity by $Au_{\varepsilon}$ and integrate by parts over $\Omega$ to have
\begin{equation*}
\frac{1}{2}\frac{d}{dt}\int_{\Omega}|A^{\frac{1}{2}}u_{\varepsilon}|^2+\int_{\Omega}|Au_{\varepsilon}|^2=\int_{\Omega}\mathcal {P}H_{\varepsilon}\cdot Au_{\varepsilon}\quad\mbox{for all}\ t\in(0, T_{\max,\varepsilon})
\end{equation*}
with $H_{\varepsilon}(x,t):=n_{\varepsilon}\nabla\Phi-\kappa(Y_{\varepsilon}u_{\varepsilon}\cdot\nabla)u_{\varepsilon}$, where we have used
\begin{equation}\label{global exi.5}
\int_{\Omega}\phi\cdot A\phi=\int_{\Omega}|A^{\frac{1}{2}}\phi|^2=\int_{\Omega}|\nabla\phi|^2\quad\mbox{for all}\ \phi\in D(A).
\end{equation}
Applying Young's inequality, $\|\mathcal {P}\phi\|_{L^{2}(\Omega)}\leq\|\phi\|_{L^{2}(\Omega)}$ for all $\phi\in L^{2}(\Omega)$ (\cite[Lemma 2.5.2]{Sohr-2001}), (\ref{global exi.4}) and (\ref{D3}), we can estimate
\begin{eqnarray*}
\int_{\Omega}\mathcal {P}H_{\varepsilon}\cdot Au_{\varepsilon}
         &\leq&\int_{\Omega}|Au_{\varepsilon}|^2+\frac{1}{4}\int_{\Omega}|\mathcal {P}H_{\varepsilon}|^2\\
         &\leq&\int_{\Omega}|Au_{\varepsilon}|^2+\frac{|\kappa|}{2}\int_{\Omega}|(Y_{\varepsilon}u_{\varepsilon}\cdot\nabla)u_{\varepsilon}|^2
               +\frac{1}{2}\int_{\Omega}|n_{\varepsilon}\nabla\Phi|^2\\
         &\leq&\int_{\Omega}|Au_{\varepsilon}|^2+C\left(\int_{\Omega}|\nabla u_{\varepsilon}|^2+\int_{\Omega}n_{\varepsilon}^2\right)
               \quad\mbox{for all}\ t\in(0, T_{\max,\varepsilon}).
\end{eqnarray*}
Hence we get
\begin{equation*}
\frac{1}{2}\frac{d}{dt}\int_{\Omega}|A^{\frac{1}{2}}u_{\varepsilon}|^2\leq C\left(\int_{\Omega}|\nabla u_{\varepsilon}|^2+\int_{\Omega}n_{\varepsilon}^2\right)\quad\mbox{for all}\ t\in(0, T_{\max,\varepsilon}).
\end{equation*}
This, along with (\ref{global exi.p}) and (\ref{global exi.5}), gives
\begin{equation*}
\int_{\Omega}|\nabla u_{\varepsilon}|^2\leq C\quad\mbox{for all}\ t\in(0, T_{\max,\varepsilon}).
\end{equation*}
We thereby obtain
\begin{equation}\label{global exi.6}
\|\mathcal {P}H_{\varepsilon}(t)\|_{L^{2}(\Omega)}\leq C\quad\mbox{for all}\ t\in(0, T_{\max,\varepsilon}).
\end{equation}
Applying the fractional power $A^{\alpha}$ with $\alpha\in(\frac{3}{4}, 1)$ to both sides of the variation-of-constants formula
\begin{equation*}
u_{\varepsilon}(t)=e^{-tA}u_{0\varepsilon}+\int_0^te^{-(t-s)A}\mathcal {P}H_{\varepsilon}(s)ds\quad\mbox{for all}\ t\in(0, T_{\max,\varepsilon}),
\end{equation*}
using the well-known smoothing estimate of the Stokes semigroup (\cite{Giga-JDE-1986}) and (\ref{global exi.6}), we have
\begin{eqnarray*}
\|A^{\alpha}u_{\varepsilon}(t)\|_{L^{2}(\Omega)}&\leq&\|A^{\alpha}e^{-tA}u_{0\varepsilon}\|_{L^{2}(\Omega)}+\int_0^t\|A^{\alpha}e^{-(t-s)A}\mathcal {P}H_{\varepsilon}(s)\|_{L^{2}(\Omega)}ds\\
   &\leq&Ct^{-\alpha}\|u_{0\varepsilon}\|_{L^{2}(\Omega)}+C\int_0^t(t-s)^{-\alpha}\|\mathcal {P}H_{\varepsilon}(s)\|_{L^{2}(\Omega)}ds\\
   &\leq&C   \quad\mbox{for all}\ t\in(\tau, T_{\max,\varepsilon})
\end{eqnarray*}
with any $\tau\in(0, T_{\max,\varepsilon})$. In view of the embedding $D(A^{\alpha})\hookrightarrow L^{\infty}(\Omega)$ asserted by our chioce of $\alpha$ (\cite[Lemma 2.4.3]{Sohr-2001}), we deduce
\begin{equation}\label{global bound for u}
\|u_{\varepsilon}(t)\|_{L^{\infty}(\Omega)}\leq C\|A^{\alpha}u_{\varepsilon}(t)\|_{L^{2}(\Omega)}\leq C\quad\mbox{for all}\ t\in(\tau, T_{\max,\varepsilon}).
\end{equation}
Let $r:=\min\{2(m+1), 4\}$, then $r>3$ due to $m\geq\frac{2}{3}$.
Employing $\nabla$ to both sides of the variation-of-constants formula for $c_{\varepsilon}$
\begin{equation*}
c_{\varepsilon}(t)=e^{(t-\frac{\tau}{2})\Delta}c_{\varepsilon}(\frac{\tau}{2})-\int_{\frac{\tau}{2}}^te^{(t-s)\Delta}\left(F_{\varepsilon}(n_{\varepsilon})f(c_{\varepsilon})+u_{\varepsilon}\cdot\nabla c_{\varepsilon}\right)(s)ds\quad\mbox{for all}\ t\in(\frac{\tau}{2}, T_{\max,\varepsilon}),
\end{equation*}
recalling the standard smoothing estimates of Neumann heat semigroup (\cite[Lemma 1.3]{Winkler-JDE-2010}, see also \cite{QS}), (\ref{RF3}), (\ref{est. c L infite}), (\ref{global bound for u}), (\ref{global exi.p}), (\ref{global exi.2}) and the H\"older inequality we obtain
\begin{eqnarray}\label{global exi.7}
\left\|\nabla c_{\varepsilon}(t)\right\|_{L^r(\Omega)}&\leq&\left\|\nabla e^{(t-\frac{\tau}{2})\Delta}c_{\varepsilon}(\frac{\tau}{2})\right\|_{L^r(\Omega)}
            +\int_{\frac{\tau}{2}}^t\left\|\nabla e^{(t-s)\Delta}\left(F_{\varepsilon}(n_{\varepsilon})f(c_{\varepsilon})+u_{\varepsilon}\cdot\nabla c_{\varepsilon}\right)(s)\right\|_{L^r(\Omega)}ds\nonumber\\
   &\leq&C\left(t-\frac{\tau}{2}\right)^{-\frac{1}{2}}\left\|c_{\varepsilon}(\frac{\tau}{2})\right\|_{L^r(\Omega)}
            +C\int_{\frac{\tau}{2}}^t(t-s)^{-\frac{1}{2}}\left\|(F_{\varepsilon}(n_{\varepsilon})f(c_{\varepsilon}))(s)\right\|_{L^r(\Omega)}ds
              \nonumber\\
            &&\quad+C\int_{\frac{\tau}{2}}^t(t-s)^{-\frac{1}{2}}\left\|(u_{\varepsilon}\cdot\nabla c_{\varepsilon})(s)\right\|_{L^r(\Omega)}ds\nonumber\\
   &\leq&C\tau^{-\frac{1}{2}}+C\int_{\frac{\tau}{2}}^t(t-s)^{-\frac{1}{2}}\left\|n_{\varepsilon}(s)\right\|_{L^r(\Omega)}ds
              +C\int_{\frac{\tau}{2}}^t(t-s)^{-\frac{1}{2}}\left\|\nabla c_{\varepsilon}(s)\right\|_{L^r(\Omega)}ds\nonumber\\
   &\leq&C\tau^{-\frac{1}{2}}+C\int_{0}^t(t-s)^{-\frac{1}{2}}ds+C\int_{0}^t(t-s)^{-\frac{1}{2}}\left\|\nabla
            c_{\varepsilon}(s)\right\|_{L^4(\Omega)}ds\nonumber\\
   &\leq&C\tau^{-\frac{1}{2}}+CT_{\max,\varepsilon}^{\frac{1}{2}}+C\left(\int_0^{T_{\max,\varepsilon}}
            (t-s)^{-\frac{2}{3}}ds\right)^{\frac{3}{4}}\left(\int_0^{T_{\max,\varepsilon}}
            \int_{\Omega}|\nabla c_{\varepsilon}(s)|^4\right)^{\frac{1}{4}}\nonumber\\
   &\leq&C\tau^{-\frac{1}{2}}+CT_{\max,\varepsilon}^{\frac{1}{2}}+CT_{\max,\varepsilon}^{\frac{1}{4}}\nonumber\\
   &\leq&C\quad\mbox{for all}\ t\in(\tau, T_{\max,\varepsilon}).
\end{eqnarray}
We next rewrite the variation-of-constants formula for $c_{\varepsilon}$ in the form
\begin{equation*}
c_{\varepsilon}(t)=e^{t(\Delta-1)}c_{0,\varepsilon}+\int_{0}^te^{(t-s)(\Delta-1)}\left(c_{\varepsilon}-
F_{\varepsilon}(n_{\varepsilon})f(c_{\varepsilon})-u_{\varepsilon}\cdot\nabla c_{\varepsilon}\right)(s)ds\quad\mbox{for all}\ t\in(0, T_{\max,\varepsilon}),
\end{equation*}
Picking $\theta\in(\frac{1}{2}+\frac{3}{2r}, 1)$, then the domain of the fractional power $D((-\Delta+1)^{\theta})\hookrightarrow W^{1,\infty}(\Omega)$ (\cite{Winkler-MN-2010,Zhang&Li-ZAMP-2013}). Hence, we obtain by virtue of $L^p$-$L^q$ estimates associated heat semigroup (\cite{Winkler-MN-2010}), (\ref{est. c L infite}), (\ref{RF3}), (\ref{D3}), (\ref{global exi.p}), (\ref{global bound for u}) and (\ref{global exi.7})
\begin{eqnarray}\label{global exi.8}
\|c_{\varepsilon}(t)\|_{W^{1,\infty}(\Omega)}&\leq&C\left\|(-\Delta+1)^{\theta}c_{\varepsilon}(t)\right\|_{L^r(\Omega)}\nonumber\\
             &\leq&Ct^{-\theta}e^{-\nu t}\left\|c_{0,\varepsilon}\right\|_{L^r(\Omega)}\nonumber\\
             &&\quad+C\int_0^t(t-s)^{-\theta}
                   e^{-\nu (t-s)}\left\|\left(c_{\varepsilon}-F_{\varepsilon}(n_{\varepsilon})f(c_{\varepsilon})-u_{\varepsilon}\cdot\nabla c_{\varepsilon}\right)(s)\right\|_{L^r(\Omega)}ds\nonumber\\
             &\leq&C\tau^{-\theta}+C\int_0^t(t-s)^{-\theta}e^{-\nu(t-s)}ds+C\int_0^t(t-s)^{-\theta}
                   e^{-\nu (t-s)}\left\|n_{\varepsilon}(s)\right\|_{L^r(\Omega)}ds\nonumber\\
             &&\quad+C\int_0^t(t-s)^{-\theta}
                   e^{-\nu (t-s)}\left\|\nabla c_{\varepsilon}(s)\right\|_{L^r(\Omega)}ds\nonumber\\
             &\leq&C\tau^{-\theta}+C\Gamma(1-\theta)\nonumber\\
             &\leq&C\quad\mbox{for all}\ t\in(\tau, T_{\max,\varepsilon})
\end{eqnarray}
with some $\nu>0$, where $\Gamma(\cdot)$ is the Gamma function. We may then apply (\ref{global exi.3}) once more and Young's inequality to deduce that
\begin{eqnarray*}
&&\frac{d}{dt}\int_{\Omega}n_{\varepsilon}^p+p(p-1)D_1\int_{\Omega}n_{\varepsilon}^{m+p-3}|\nabla n_{\varepsilon}|^2\\
&&\leq C\int_{\Omega}n_{\varepsilon}^{p-2}\nabla c_{\varepsilon}\cdot\nabla n_{\varepsilon}\\
&&\leq C\int_{\Omega}n_{\varepsilon}^{p-2}|\nabla n_{\varepsilon}|\\
&&\leq p(p-1)D_1\int_{\Omega}n_{\varepsilon}^{m+p-3}|\nabla n_{\varepsilon}|^2+\int_{\Omega}n_{\varepsilon}^{p-m-1}+C\\
&&\leq p(p-1)D_1\int_{\Omega}n_{\varepsilon}^{m+p-3}|\nabla n_{\varepsilon}|^2+\int_{\Omega}n_{\varepsilon}^{p}+C
       \quad\mbox{for all}\ t\in(\tau, T_{\max,\varepsilon}).
\end{eqnarray*}
Therefore, integrating with respect to $t$, we obtain
\begin{equation*}
\int_{\Omega}n_{\varepsilon}^p\leq C\quad\mbox{for all}\ t\in(\tau, T_{\max,\varepsilon})
\end{equation*}
with any $p\geq1$. Upon an application of the well-known Moser-Alikakos iteration procedure (\cite{Alikakos-CPDE-1979,Tao&Winkler-JDE-2012}), we see that
\begin{equation}\label{global bound for n}
\|n_{\varepsilon}(t)\|_{L^{\infty}(\Omega)}\leq C\quad\mbox{for all}\ t\in(\tau, T_{\max,\varepsilon}).
\end{equation}
In view of (\ref{global bound for u}) (\ref{global bound for n}), we apply Lemma \ref{Lem Local sol.} to reach a contradiction.
\end{proof}

\section{Further $\varepsilon$-independent estimates for (\ref{RegularizedCNS})}
In order to pass to limits in (\ref{RegularizedCNS}) with safety, we need some more $\varepsilon$-independent estimates for the solution.
\begin{lem}\label{Lem spati-temporal est.}
Suppose that (\ref{D1})-(\ref{D6}) hold. There exists $C>0$ such that for all $\varepsilon\in(0,1)$, the solutions of (\ref{RegularizedCNS}) satisfy
\begin{equation}\label{st est11}
\int_0^T\int_{\Omega}\left|\nabla n_{\varepsilon}^{\frac{m}{2}}\right|^2\leq C(T+1),\quad\mbox{for all}\ T>0,
\end{equation}
\begin{equation}\label{st est1}
\int_0^T\int_{\Omega}(n_{\varepsilon}+\varepsilon)^{p}\leq C(T+1),\quad 1\leq p\leq\frac{3m+2}{3}\quad\mbox{for all}\ T>0,
\end{equation}
\begin{equation}\label{st est2}
\int_0^T\int_{\Omega}\frac{D_{\varepsilon}(n_{\varepsilon})}{n_{\varepsilon}}|\nabla n_{\varepsilon}|^2\leq C(T+1)\quad\mbox{for all}\ T>0,
\end{equation}
\begin{equation}\label{st est21}
\int_0^T\int_{\Omega}\left(D_{\varepsilon}(n_{\varepsilon})\nabla n_{\varepsilon}\right)^{\frac{3m+2}{3m+1}}\leq C(T+1)\quad\mbox{for all}\ T>0
\end{equation}
and
\begin{equation}\label{st est3}
\int_0^T\int_{\Omega}|u_{\varepsilon}|^{\frac{10}{3}}\leq C(T+1)\quad\mbox{for all}\ T>0.
\end{equation}
Moreover, if $\frac{2}{3}\leq m\leq2$, then we have
\begin{equation}\label{st est4}
\int_0^T\int_{\Omega}\left|\nabla n_{\varepsilon}\right|^{\frac{3m+2}{4}}\leq C(T+1),\quad\mbox{for all}\ T>0.
\end{equation}
\end{lem}
\begin{proof}
From Lemma \ref{Lem energy I2} we know that there exists $C_1>0$ such that
\begin{equation}\label{est. m2}
\int_0^T\int_{\Omega}(n_{\varepsilon}+\varepsilon)^{m-2}|\nabla n_{\varepsilon}|^2\leq C_1(T+1)\quad\mbox{for all}\ T>0.
\end{equation}
Then, (\ref{st est11}) is a direct consequence of (\ref{est. m2}). Due to the fact that $\Omega$ is bounded, we only need to prove (\ref{st est1}) with $p=\frac{3m+2}{3}$. We employ the Gagliardo-Nirenberg inequality to find $C_2>0$ and $C_3>0$ fulfilling
\begin{eqnarray*}
\int_0^T\int_{\Omega}(n_{\varepsilon}+\varepsilon)^{\frac{3m+2}{3}}&=&\int_0^T\left\|(n_{\varepsilon}+\varepsilon)^{\frac{m}{2}}\right\|
                                                        ^{\frac{2(3m+2)}{3m}}_{L^{\frac{2(3m+2)}{3m}}(\Omega)}\nonumber\\
           &\leq&C_2\left(\left\|\nabla (n_{\varepsilon}+\varepsilon)^{\frac{m}{2}}\right\|^{\frac{3m}{3m+2}}_{L^{2}(\Omega)}\cdot
                     \left\|(n_{\varepsilon}+\varepsilon)^{\frac{m}{2}}\right\|^{\frac{2}{3m+2}}_{L^{\frac{2}{m}}(\Omega)}
                     +\left\|(n_{\varepsilon}+\varepsilon)^{\frac{m}{2}}\right\|_{L^{\frac{2}{m}}(\Omega)}\right)^{\frac{2(3m+2)}{3m}}\nonumber\\
           &=&C_2\left(\left\|\nabla (n_{\varepsilon}+\varepsilon)^{\frac{m}{2}}\right\|^{2}_{L^{2}(\Omega)}\cdot
                     \left\|(n_{\varepsilon}+\varepsilon)^{\frac{m}{2}}\right\|^{\frac{4}{3m}}_{L^{\frac{2}{m}}(\Omega)}
                     +\left\|(n_{\varepsilon}+\varepsilon)^{\frac{m}{2}}\right\|^{\frac{2(3m+2)}{3m}}_{L^{\frac{2}{m}}(\Omega)}\right)\nonumber\\
           &\leq&C_3(T+1)\quad\mbox{for all}\ T>0.
\end{eqnarray*}
Recalling the proof in Lemma \ref{Lem energy I1}, we obtain
\begin{equation*}
y'_{\varepsilon}(t)+\frac{1}{2K}\int_{\Omega}\frac{D_{\varepsilon}(n_{\varepsilon})}{n_{\varepsilon}}|\nabla n_{\varepsilon}|^2\leq K\quad\mbox{for all}\ T>0,
\end{equation*}
where $y_{\varepsilon}$ and $K$ are provided by (\ref{y definition}) and Lemma \ref{Lem energy I}, respectively. Integrating this in time over $(0, T)$ yields (\ref{st est2}). By H\"older's inequality, (\ref{st est1}) and (\ref{st est2}), we can find $C_4>0$ such that
\begin{eqnarray*}
\int_0^T\int_{\Omega}\left(D_{\varepsilon}(n_{\varepsilon})\nabla n_{\varepsilon}\right)^{\frac{3m+2}{3m+1}}
  &\leq&\left(\int_0^T\int_{\Omega}\frac{D_{\varepsilon}(n_{\varepsilon})}{n_{\varepsilon}}|\nabla n_{\varepsilon}|^2\right)^{\frac{3m+2}{6m+2}}
     \left(\int_0^T\int_{\Omega}(D_{\varepsilon}(n_{\varepsilon})n_{\varepsilon})^{\frac{3m+2}{3m}}\right)^{\frac{3m}{6m+2}}\\
  &\leq&D_2^{\frac{3m+2}{6m+2}}\left(\int_0^T\int_{\Omega}\frac{D_{\varepsilon}(n_{\varepsilon})}{n_{\varepsilon}}|\nabla n_{\varepsilon}|^2\right)^{\frac{3m+2}{6m+2}}
     \left(\int_0^T\int_{\Omega}(n_{\varepsilon}+\varepsilon)^{\frac{3m+2}{3}}\right)^{\frac{3m}{6m+2}}\\
  &\leq&C_4(T+1)\quad\mbox{for all}\ T>0.
\end{eqnarray*}
Using the Gagliardo-Nirenberg inequality and Lemma \ref{Lem energy I2} again we have
\begin{eqnarray*}
\int_0^T\int_{\Omega}|u_{\varepsilon}|^{\frac{10}{3}}
     &=&\int_0^T\left\|u_{\varepsilon}(t)\right\|^{\frac{10}{3}}_{L^{\frac{10}{3}}(\Omega)}\\
     &\leq&C_5\int_0^T\left(\|\nabla u_{\varepsilon}(t)\|^{2}_{L^{2}(\Omega)}\cdot\|u_{\varepsilon}(t)\|^{\frac{4}{3}}_{L^{2}(\Omega)}
           +\|u_{\varepsilon}(t)\|^{\frac{10}{3}}_{L^{2}(\Omega)}\right)\\
     &\leq&C_6(T+1)\quad\mbox{for all}\ T>0
\end{eqnarray*}
with $C_5>0$ and $C_6>0$. Finally, we prove (\ref{st est4}). Since $\frac{2}{3}\leq m\leq2$, applying (\ref{est. m2}) and Young's inequality we get $C_7>0$ and $C_8>0$ such that
\begin{eqnarray*}
\int_0^T\int_{\Omega}\left|\nabla n_{\varepsilon}\right|^{\frac{3m+2}{4}}&=&\int_0^T\int_{\Omega}n^{\frac{(m-2)(3m+2)}{8}}_{\varepsilon}
\left|\nabla n_{\varepsilon}\right|^{\frac{3m+2}{4}}n^{\frac{(2-m)(3m+2)}{8}}_{\varepsilon}\\
&\leq&C_7\left(\int_0^T\int_{\Omega}n_{\varepsilon}^{m-2}|\nabla n_{\varepsilon}|^2+\int_0^T\int_{\Omega}n^{\frac{3m+2}{3}}_{\varepsilon}\right)\\
&\leq& C_8(T+1),\quad\mbox{for all}\ T>0.
\end{eqnarray*}
This completes the proof.
\end{proof}

We derive an $L^p$-bound for $n_{\varepsilon}+\varepsilon$ and a estimate of space-time integral $\int_0^T\int_{\Omega}|\nabla (n_{\varepsilon}+\varepsilon)^{\bar{r}}|^2$ as a supplement to the regularity property concerning $n_{\varepsilon}$ in the case $m>2$.

\begin{lem}\label{Lem Lp of u}
Let $m>\frac{10}{9}$. For all $\varepsilon\in(0,1)$, there exists $C(T)>0$ and $C>0$ such that
\begin{equation}\label{est. u Lp}
\int_{\Omega}(n_{\varepsilon}+\varepsilon)^{p}\leq C(T)\quad\mbox{for all}\ t>0.
\end{equation}
with $1\leq p<9(m-1)$ and
\begin{equation}\label{est. nabla u Lp}
\int_{\Omega}\left|\nabla(n_{\varepsilon}+\varepsilon)^{\bar{r}}\right|^2\leq C(T+1)\quad\mbox{for all}\ T>0.
\end{equation}
with $\frac{m}{2}<\bar{r}<5(m-1)$.
\end{lem}
\begin{proof}
It is based on a bootstrap argument (\cite{Tao&Winkler-AIHP-2013}, see also \cite{Horstmann&Winkler-JDE-2005,TW-CPDE-2007}). We multiply the first equation in (\ref{RegularizedCNS}) by $p (n_{\varepsilon}+\varepsilon)^{p-1}$ to deduce that
\begin{eqnarray}\label{Lp of u1}
&&\frac{d}{dt}\int_{\Omega}(n_{\varepsilon}+\varepsilon)^p+p(p-1)\int_{\Omega}(n_{\varepsilon}
+\varepsilon)^{p-2}D_{\varepsilon}(n_{\varepsilon})|\nabla n_{\varepsilon}|^2\nonumber\\
&&\quad=p(p-1)\int_{\Omega}n_{\varepsilon}(n_{\varepsilon}+\varepsilon)^{p-2}F'_{\varepsilon}(n_{\varepsilon})\chi(c_{\varepsilon})\nabla c_{\varepsilon}\cdot \nabla n_{\varepsilon}
\end{eqnarray}
for all $t>0$. However, unlike the proof of Lemma \ref{Lem global exi.} we deal with $n_{\varepsilon}F'_{\varepsilon}(n_{\varepsilon})$ together, because our goal is to get an $\varepsilon$-independent bound (\ref{est. u Lp}). More precisely, from (\ref{Lp of u1}) and (\ref{RF1}) we have
\begin{eqnarray}\label{Lp of u2}
&&\frac{d}{dt}\int_{\Omega}(n_{\varepsilon}+\varepsilon)^p+p(p-1)D_1\int_{\Omega}(n_{\varepsilon}+\varepsilon)^{m+p-3}|\nabla n_{\varepsilon}|^2\nonumber\\
&&\quad\leq p(p-1)\max_{s\in[0,M]}\chi(s)\int_{\Omega}(n_{\varepsilon}+\varepsilon)^{p-1}\nabla c_{\varepsilon}\cdot
        \nabla n_{\varepsilon}
\end{eqnarray}
for all $t>0$. Applying the H\"older and Young inequalities in the right-hand side of (\ref{Lp of u2}), we have $C_1>0$ such that
\begin{eqnarray}\label{Lp of u3}
&&\frac{d}{dt}\int_{\Omega}(n_{\varepsilon}+\varepsilon)^p+p(p-1)D_1\int_{\Omega}(n_{\varepsilon}+\varepsilon)^{m+p-3}|\nabla n_{\varepsilon}|^2\nonumber\\
&&\quad\leq p(p-1)\max_{s\in[0,M]}\chi(s)\left\|(n_{\varepsilon}+\varepsilon)^{\frac{m+p-3}{2}}\nabla n_{\varepsilon}\right\|_{L^{2}(\Omega)}\cdot
      \left\|(n_{\varepsilon}+\varepsilon)^{\frac{p-m+1}{2}}\right\|_{L^{4}(\Omega)}\cdot\|\nabla c_{\varepsilon}\|_{L^{4}(\Omega)} \nonumber\\
&&\quad\leq \frac{p(p-1)D_1}{2}\int_{\Omega}(n_{\varepsilon}+\varepsilon)^{m+p-3}|\nabla n_{\varepsilon}|^2+C_1
      \left\|(n_{\varepsilon}+\varepsilon)^{\frac{p-m+1}{2}}\right\|^2_{L^{4}(\Omega)}\cdot\|\nabla c_{\varepsilon}\|^2_{L^{4}(\Omega)}
\end{eqnarray}
for all $t>0$. Assume that
\begin{equation*}
\int_{\Omega}(n_{\varepsilon}+\varepsilon)^{p_i}\leq C(T)\quad\mbox{for all}\ t>0
\end{equation*}
holds with some $p_i\geq1$ (this is true for $p_1:=1$ by (\ref{est. u L1}) and $\varepsilon<1$). The Gagliardo-Nirenberg inequality gives $C_2>0$ such that
\begin{eqnarray*}
\left\|(n_{\varepsilon}+\varepsilon)^{\frac{p-m+1}{2}}\right\|^2_{L^{4}(\Omega)}
&=&\left\|(n_{\varepsilon}+\varepsilon)^{\frac{p+m-1}{2}}\right\|^{\frac{2(p-m+1)}{p+m-1}}_{L^{\frac{4(p-m+1)}{p+m-1}}(\Omega)}\\
&\leq&C_2\Bigg{(}\left\|\nabla (n_{\varepsilon}+\varepsilon)^{\frac{p+m-1}{2}}\right\|^{\alpha}_{L^{2}(\Omega)}\cdot
      \left\|(n_{\varepsilon}+\varepsilon)^{\frac{p+m-1}{2}}\right\|^{1-\alpha}_{L^{\frac{2p_i}{p+m-1}}(\Omega)}\\
&\quad&+
      \left\|(n_{\varepsilon}+\varepsilon)^{\frac{p+m-1}{2}}\right\|_{L^{\frac{2p_i}{p+m-1}}(\Omega)}\Bigg{)}^{\frac{2(p-m+1)}{p+m-1}}
\end{eqnarray*}
for all $t>0$, where
\begin{equation*}
\alpha=\frac{1-\frac{p_i}{2(p-m+1)}}{1-\frac{p_i}{3(p+m-1)}}\in(0, 1).
\end{equation*}
If
\begin{equation*}
\frac{2(p-m+1)}{p+m-1}\alpha=1,
\end{equation*}
which is equivalent to $p=3(m-1)+\frac{2}{3}p_i$, and therefore by Young's inequality
\begin{eqnarray*}
&&C_1\left\|(n_{\varepsilon}+\varepsilon)^{\frac{p-m+1}{2}}\right\|^2_{L^{4}(\Omega)}\cdot\|\nabla c_{\varepsilon}\|^2_{L^{4}(\Omega)}\\
&&\quad\leq C_3\left(\left\|\nabla (n_{\varepsilon}+\varepsilon)^{\frac{p+m-1}{2}}\right\|_{L^{2}(\Omega)}
      +1\right)\|\nabla c_{\varepsilon}\|^2_{L^{4}(\Omega)}\\
&&\quad\leq \frac{p(p-1)D_1}{4}\int_{\Omega}(n_{\varepsilon}+\varepsilon)^{m+p-3}|\nabla n_{\varepsilon}|^2+C_4\|\nabla c_{\varepsilon}\|^4_{L^{4}(\Omega)}+C_5
\end{eqnarray*}
for all $t>0$ with certain positive constants $C_3$, $C_4$ and $C_5$. Substituting this into (\ref{Lp of u3}), we obtain
\begin{eqnarray}\label{ODI n+narepsilon}
\frac{d}{dt}\int_{\Omega}(n_{\varepsilon}+\varepsilon)^p+\frac{p(p-1)D_1}{4}\int_{\Omega}(n_{\varepsilon}+\varepsilon)^{m+p-3}|\nabla n_{\varepsilon}|^2\leq C_4\|\nabla c_{\varepsilon}\|^4_{L^{4}(\Omega)}+C_5
\end{eqnarray}
for all $t>0$. By integration, we finally get
\begin{equation*}
\int_{\Omega}(n_{\varepsilon}+\varepsilon)^{p}\leq C(T)\quad\mbox{for all}\ t>0.
\end{equation*}
with $p=3(m-1)+\frac{2}{3}p_i$. By this iterative procedure, there exists a sequence $\{p_i\}$ such that
\begin{equation*}
p_{i+1}=3(m-1)+\frac{2}{3}p_i.
\end{equation*}
It is easy to check that the sequence $\{p_i\}$ is increasing and $p_i\rightarrow 9(m-1)$ as $i\rightarrow\infty$. Therefore, we can reach any $p<9(m-1)$ by finite steps and (\ref{est. u Lp}) is thereby proved. Another integration of (\ref{ODI n+narepsilon}) yields
\begin{eqnarray*}
\int_0^T\int_{\Omega}\left|\nabla(n_{\varepsilon}+\varepsilon)^{\frac{m+p-1}{2}}\right|^2&\leq& C_6\left(\int_{\Omega}(n_{0\varepsilon}+1)^p+\int_0^T\int_{\Omega}|\nabla c_{\varepsilon}|^4+1\right)\\
&\leq& C_7(T+1)\quad\mbox{for all}\ T>0
\end{eqnarray*}
with $C_6>0$ and $C_7>0$. This proves (\ref{est. nabla u Lp}).
\end{proof}

In order to derive strong compactness properties, we also need some estimates concerning the time derivative of the solution.
\begin{lem}\label{Lem time derivative}
Let $\gamma:=\max\{1, \frac{m}{2}\}$. There exists $C>0$ such that for all $\varepsilon\in(0,1)$ we have
\begin{equation}\label{time derivative est1}
\int_0^T\left\|\frac{\partial}{\partial t}n_{\varepsilon}^{\gamma}\right\|_{({W^{2,q}(\Omega)})^{*}}dt\leq C(T+1)\quad\mbox{for all}\ T>0
\end{equation}
with some $q>3$, and
\begin{equation}\label{time derivative est2}
\int_0^T\left\|\frac{\partial}{\partial t}\sqrt{c_{\varepsilon}}\right\|^{\frac{5}{3}}_{({W^{1,\frac{5}{2}}(\Omega)})^{*}}dt\leq C(T+1)\quad\mbox{for all}\ T>0
\end{equation}
as well as
\begin{equation}\label{time derivative est3}
\int_0^T\left\|\frac{\partial}{\partial t}u_{\varepsilon }\right\|^{\frac{5}{4}}_{({W_{0,\sigma}^{1,\frac{5}{2}}(\Omega)})^{*}}dt\leq C(T+1)\quad\mbox{for all}\ T>0.
\end{equation}
\end{lem}
\begin{proof}
Multiplying the first equation in (\ref{RegularizedCNS}) by $\gamma n_{\varepsilon}^{\gamma-1}\varphi$ with $\varphi\in C^{\infty}(\bar{\Omega})$ and integrating by parts, we obtain
\begin{eqnarray*}
\left|\int_{\Omega}(n^{\gamma}_{\varepsilon})_t\varphi\right|&=&\bigg{|}-\gamma(\gamma-1)\int_{\Omega}D_{\varepsilon}(n_{\varepsilon})
          n^{\gamma-2}_{\varepsilon}|\nabla n_{\varepsilon}|^2\varphi -\gamma\int_{\Omega}D_{\varepsilon}(n_{\varepsilon})
          n^{\gamma-1}_{\varepsilon}\nabla n_{\varepsilon}\cdot \nabla\varphi\\
  &\quad&+\gamma(\gamma-1)\int_{\Omega}n^{\gamma-1}_{\varepsilon}F'_{\varepsilon}(n_{\varepsilon})\chi(c_{\varepsilon})\nabla c_{\varepsilon}\cdot\nabla n_{\varepsilon}\varphi
  \\&\quad&+\gamma\int_{\Omega}n^{\gamma}_{\varepsilon}F'_{\varepsilon}(n_{\varepsilon})
  \chi(c_{\varepsilon})\nabla c_{\varepsilon}\cdot\nabla\varphi
    +\int_{\Omega}n^{\gamma}_{\varepsilon}u_{\varepsilon}\cdot\nabla\varphi\bigg{|}\\
    &\leq&\bigg{(}\gamma(\gamma-1)\int_{\Omega}D_{\varepsilon}(n_{\varepsilon})
          n^{\gamma-2}_{\varepsilon}|\nabla n_{\varepsilon}|^2+\gamma\int_{\Omega}D_{\varepsilon}(n_{\varepsilon})
          n^{\gamma-1}_{\varepsilon}|\nabla n_{\varepsilon}|\\
  &\quad&+\gamma(\gamma-1)\int_{\Omega}|n^{\gamma-1}_{\varepsilon}F'_{\varepsilon}(n_{\varepsilon})\chi(c_{\varepsilon})\nabla c_{\varepsilon}\cdot\nabla n_{\varepsilon}|\\
  &\quad&+\gamma\int_{\Omega}|n^{\gamma}_{\varepsilon}F'_{\varepsilon}(n_{\varepsilon})
  \chi(c_{\varepsilon})\nabla c_{\varepsilon}|+\int_{\Omega}n^{\gamma}_{\varepsilon}|u_{\varepsilon}|\bigg{)}\|\varphi\|_{W^{1,\infty}(\Omega)}
\end{eqnarray*}
for all $t>0$. Due to the embedding $W^{2,q}(\Omega)\hookrightarrow W^{1,\infty}(\Omega)$ for $q>3$, we deduce $C_1>0$ such that
\begin{eqnarray}\label{time E1}
&&\int_0^T\left\|(n^{\gamma}_{\varepsilon})_t\right\|_{({W^{2,q}(\Omega)})^{*}}dt\nonumber\\
&&\quad\leq C_1\bigg{(}\gamma(\gamma-1)\int_0^T\int_{\Omega}D_{\varepsilon}(n_{\varepsilon})
          n^{\gamma-2}_{\varepsilon}|\nabla n_{\varepsilon}|^2+\gamma\int_0^T\int_{\Omega}D_{\varepsilon}(n_{\varepsilon})
          n^{\gamma-1}_{\varepsilon}|\nabla n_{\varepsilon}|\nonumber\\
&&\quad+\gamma(\gamma-1)\int_0^T\int_{\Omega}|n^{\gamma-1}_{\varepsilon}F'_{\varepsilon}(n_{\varepsilon})\chi(c_{\varepsilon})
          \nabla c_{\varepsilon}\cdot\nabla n_{\varepsilon}|\nonumber\\
&&\quad+\gamma\int_0^T\int_{\Omega}|n^{\gamma}_{\varepsilon}F'_{\varepsilon}(n_{\varepsilon})
          \chi(c_{\varepsilon})\nabla c_{\varepsilon}|+\int_0^T\int_{\Omega}n^{\gamma}_{\varepsilon}|u_{\varepsilon}|\bigg{)}\nonumber\\
&&=:I_1+I_2+I_3+I_4+I_5
\end{eqnarray}
for all $T>0$. By (\ref{D2}), (\ref{est. nabla u Lp}), (\ref{st est11}) and Young's inequality, we can find $C_2>0$ such that
\begin{eqnarray}\label{time E10}
I_1
&\leq&D_2\gamma(\gamma-1)\int_0^T\int_{\Omega}(n_{\varepsilon}+\varepsilon)^{m}n^{\gamma-1}_{\varepsilon}|\nabla n_{\varepsilon}|^2\nonumber\\
&\leq&\frac{D_2(\gamma-1)}{m+1}\int_0^T\int_{\Omega}\nabla(n_{\varepsilon}+\varepsilon)^{m+1}\cdot \nabla n^{\gamma}_{\varepsilon}\nonumber\\
&\leq&\frac{D_2(\gamma-1)}{2(m+1)}\left(\int_0^T\int_{\Omega}\left|\nabla(n_{\varepsilon}+\varepsilon)^{m+1}\right|^2
      +\int_0^T\int_{\Omega}\left|\nabla n^{\gamma}_{\varepsilon}\right|^2\right)\nonumber\\
&\leq& C_2(T+1).
\end{eqnarray}
Employing (\ref{D2}), (\ref{st est2}) and Young's inequality we have $C_3>0$ such that
\begin{eqnarray}\label{time E11}
I_2&\leq&\frac{\gamma}{2}\int_0^T\int_{\Omega}\frac{D_{\varepsilon}(n_{\varepsilon})}{n_{\varepsilon}}|\nabla n_{\varepsilon}|^2
          +\frac{\gamma}{2}\int_0^T\int_{\Omega}D_{\varepsilon}(n_{\varepsilon})n^{2\gamma-1}_{\varepsilon}\nonumber\\
                &\leq&\frac{\gamma}{2}\int_0^T\int_{\Omega}\frac{D_{\varepsilon}(n_{\varepsilon})}{n_{\varepsilon}}|\nabla n_{\varepsilon}|^2
          +\frac{\gamma D_2}{2}\int_0^T\int_{\Omega}(n_{\varepsilon}+\varepsilon)^{m+2\gamma-2}\nonumber\\
                &\leq& C_3(T+1),
\end{eqnarray}
where we have used when $1\leq m\leq2$
\begin{equation*}
\int_0^T\int_{\Omega}(n_{\varepsilon}+\varepsilon)^{m+2\gamma-2}=\int_0^T\int_{\Omega}(n_{\varepsilon}+\varepsilon)^{m}\leq C_4(T+1)\quad\mbox{by}\  (\ref{st est1})
\end{equation*}
with $C_4>0$ and in the case $m>2$
\begin{equation*}
\int_0^T\int_{\Omega}(n_{\varepsilon}+\varepsilon)^{m+2\gamma-2}=\int_0^T\int_{\Omega}(n_{\varepsilon}+\varepsilon)^{2(m-1)}\leq C(T)\quad\mbox{by}\  (\ref{est. u Lp}).
\end{equation*}
From (\ref{RF1}), (\ref{est. c L infite}), (\ref{global exi.2}) and (\ref{st est11}), we estimate
\begin{eqnarray}\label{time E12}
I_3&\leq&\gamma(\gamma-1)C_5\int_0^T\int_{\Omega}|n^{\gamma-1}_{\varepsilon}\nabla n_{\varepsilon}\cdot\nabla c_{\varepsilon}|\nonumber\\
          &\leq&C_5(\gamma-1)\int_0^T\int_{\Omega}|\nabla n^{\gamma}_{\varepsilon}\cdot\nabla c_{\varepsilon}|\nonumber\\
          &\leq&\frac {C_5(\gamma-1)}{2}\left(\int_0^T\int_{\Omega}|\nabla n^{\gamma}_{\varepsilon}|^2+\frac{1}{2}\int_0^T\int_{\Omega}|\nabla c_{\varepsilon}|^4+\frac{1}{2}\right)\nonumber\\
          &\leq& C_6(T+1)
\end{eqnarray}
with $C_5>0$ and $C_6>0$. Applying (\ref{RF1}), (\ref{est. c L infite}), (\ref{global exi.2}) and (\ref{st est1}), we find $C_7>0$ and $C_8>0$ such that
\begin{eqnarray}\label{time E13}
I_4&\leq&C_7\left(\int_0^T\int_{\Omega}n^{\frac{4\gamma}{3}}_{\varepsilon}+\int_0^T\int_{\Omega}|\nabla c_{\varepsilon}|^{4}\right)\nonumber\\
  &\leq&C_8(T+1).
\end{eqnarray}
Moreover, we use (\ref{st est1}) and (\ref{st est3}) to give $C_9>0$ and $C_10>0$ fulfilling
\begin{eqnarray}\label{time E14}
I_5&\leq& C_9\left(\int_0^T\int_{\Omega}n^{\frac{10\gamma}{7}}_{\varepsilon}+\int_0^T\int_{\Omega}|u_{\varepsilon}|^{\frac{10}{3}}\right)\nonumber\\
  &\leq&C_{10}(T+1).
\end{eqnarray}
Then, (\ref{time derivative est1}) follows by combining (\ref{time E1})-(\ref{time E14}). Multiplying the second equation in (\ref{RegularizedCNS}) by $\frac{\varphi}{2\sqrt{c_{\varepsilon}}}$ with $\varphi\in C^{\infty}(\bar{\Omega})$ and the third by $\phi\in \left(C^{\infty}_{0,\sigma}(\bar{\Omega})\right)^3$, we obtain (\ref{time derivative est2}) and (\ref{time derivative est3}) in a completed similar manner (see \cite{winkler2014global} for details).
\end{proof}

\section{Global weak solutions for (\ref{CNS})-(\ref{initial c})}
We are now in the position to construct global weak solutions for (\ref{CNS})-(\ref{initial c}). Before going into details, let us first give the definition of weak solution.
\begin{defn}\label{global weak sol}
We call $(n, c, u)$ a {\it{global weak solution}} of (\ref{CNS})-(\ref{initial c}) if
\begin{equation*}
n\in L_{loc}^{1}(\bar{\Omega}\times[0, \infty)),\quad c\in L_{loc}^1([0, \infty);W^{1,1}(\Omega)),\quad u\in\big{(}L_{loc}^1([0, \infty);W_0^{1,1}(\Omega))\big{)}^3
\end{equation*}
such that $n\geq0$ and $c\geq0$ {\it a.e.} in $\Omega\times(0, \infty)$, and that
\begin{eqnarray*}
&&nf(c)\in L_{loc}^1([0, \infty);L^{1}(\Omega)),\\
&&D(n)\nabla n,\ n\chi(c)\nabla c,\ nu\ \mbox{and}\ cu\ \mbox{belong to}\left(L_{loc}^1([0, \infty);L^{1}(\Omega))\right)^3,\\
&&u\otimes u\in \left(L_{loc}^1([0, \infty);L^{1}(\Omega))\right)^{3\times3}
\end{eqnarray*}
and that
\begin{eqnarray*}
    &&\int_0^{\infty}\int_{\Omega}n_t\phi_1-\int_0^{\infty}\int_{\Omega}nu\cdot\nabla \phi_1=-\int_0^{\infty}\int_{\Omega}D(n)\nabla n\cdot\nabla\phi_1+\int_0^{\infty}\int_{\Omega}n\chi(c)\nabla c\cdot\nabla\phi_1,\\
    &&\int_0^{\infty}\int_{\Omega}c_t\phi_2-\int_0^{\infty}\int_{\Omega}cu \cdot\nabla\phi_2=-\int_0^{\infty}\int_{\Omega}\nabla c\cdot\nabla\phi_2-\int_0^{\infty}\int_{\Omega}nf(c)\phi_2,\\
    &&\int_0^{\infty}\int_{\Omega}u_t\cdot\phi_3-\kappa\int_0^{\infty}\int_{\Omega}u\otimes u\cdot\nabla\phi_3=-\int_0^{\infty}\int_{\Omega}\nabla u\cdot\nabla\phi_3+\int_0^{\infty}\int_{\Omega}n\nabla\Phi\cdot \phi_3
\end{eqnarray*}
hold for all $\phi_1\in C_0^{\infty}(\bar{\Omega}\times[0,\infty))$, $\phi_2\in C_0^{\infty}(\bar{\Omega}\times[0,\infty))$ and $\phi_3\in \left(C_0^{\infty}(\Omega\times[0,\infty))\right)^3$ satisfying $\nabla\cdot\phi_3=0$.
\end{defn}

We can now pass to the proof of our main result.

\noindent{\bf Proof of Theorem \ref{main result}.} Let $\gamma$ be given by Lemma \ref{Lem spati-temporal est.} and set
\begin{eqnarray*}
  \beta:=\left\{\begin{array}{lll}
     \medskip
     \frac{3m+2}{4},&{}1\leq m\leq2,\\
     \medskip
     2,&{}m>2.
  \end{array}\right.
\end{eqnarray*}
By Lemma \ref{Lem energy I2}, Lemma \ref{Lem spati-temporal est.} and Lemma \ref{Lem time derivative}, for some $C>0$ which is independent of $\varepsilon$, we have
\begin{eqnarray*}
&&\left\|n^{\gamma}_{\varepsilon}\right\|_{L_{loc}^{\beta}([0, \infty); W^{1,\beta}(\Omega))}\leq C(T+1),\\
&&\left\|\left(n^{\gamma}_{\varepsilon}\right)_t\right\|_{L_{loc}^{1}([0, \infty); (W^{2,q}(\Omega))^{*})}\leq C(T+1)\quad \mbox{with some}\ q>3,\\
&&\left\|\sqrt{c_{\varepsilon}}\right\|_{L^{2}_{loc}([0, \infty); W^{2,2}(\Omega))}\leq C(T+1),\\
&&\left\|(\sqrt{c_{\varepsilon}})_t\right\|_{L^{\frac{5}{3}}_{loc}([0, \infty); (W^{1,\frac{5}{2}}(\Omega))^{*})}\leq C(T+1),\\
&&\left\|u_{\varepsilon}\right\|_{L_{loc}^{2}([0, \infty); W^{1,2}(\Omega))}\leq C(T+1),\quad \mbox{and}\\
&&\left\|u_{\varepsilon t}\right\|_{L_{loc}^{\frac{5}{4}}([0, \infty); (W_{0,\sigma}^{1,\frac{5}{2}}(\Omega))^{*})}\leq C(T+1)
\end{eqnarray*}
for all $T>0$. Therefore, the Aubin-Lions lemma (\cite{Lions-1969}, see \cite{Simon-AMPA-1987} for the case involving the space $L^p$ with $p=1$) asserts that
\begin{eqnarray*}
&&(n^{\gamma}_{\varepsilon})_{\varepsilon\in(0,1)}\ \mbox{is strongly precompact in}\ L_{loc}^{\beta}(\bar{\Omega}\times[0, \infty)),\\
&&(\sqrt{c_{\varepsilon}})_{\varepsilon\in(0,1)}\ \mbox{is strongly precompact in}\ L_{loc}^{2}([0, \infty); W^{1,2}(\Omega))\quad \mbox{and}\\
&&(u_{\varepsilon})_{\varepsilon\in(0,1)}\ \mbox{is strongly precompact in}\ L_{loc}^{2}(\bar{\Omega}\times[0, \infty)).
\end{eqnarray*}
This yields a subsequence $\varepsilon:=\varepsilon_j\in(0,1)$ $(j\in\mathbb{N})$ and the limit functions $n$, $c$ and $u$ such that
\begin{eqnarray*}
n^{\gamma}_{\varepsilon}\rightarrow n^{\gamma}              &\quad& \mbox{in}\ L_{loc}^{\beta}(\bar{\Omega}\times[0, \infty)),\ \mbox{and a.e. in}\ \Omega\times(0, \infty),\\
n_{\varepsilon}\rightharpoonup n          &\quad& \mbox{in}\ L_{loc}^{\frac{3m+2}{3}}(\bar{\Omega}\times[0, \infty)),\\
D_{\varepsilon}(n_{\varepsilon})\nabla n_{\varepsilon}\rightharpoonup D(n)\nabla n
                                          &\quad& \mbox{in}\ L_{loc}^{\frac{3m+2}{3m+1}}(\bar{\Omega}\times[0, \infty))
\end{eqnarray*}
and
\begin{eqnarray*}
\sqrt{c_{\varepsilon}}\rightarrow\sqrt{c} &\quad& \mbox{in}\ L_{loc}^{2}([0, \infty); W^{1,2}(\Omega))\ \mbox{and a.e. in}\ \Omega\times(0, \infty),\\
c_{\varepsilon}\stackrel{\ast}{\rightharpoonup} c
                                          &\quad& \mbox{in}\ L^{\infty}(\Omega\times(0, \infty)),\\
\nabla c^{\frac{1}{4}}_{\varepsilon}\rightharpoonup \nabla c^{\frac{1}{4}}
                                          &\quad& \mbox{in}\ L_{loc}^{4}(\bar{\Omega}\times[0, \infty))
\end{eqnarray*}
as well as
\begin{eqnarray*}
u_{\varepsilon}\rightarrow u              &\quad& \mbox{in}\ L_{loc}^{2}(\bar{\Omega}\times[0, \infty))\ \mbox{and a.e. in}\ \Omega\times(0, \infty),\\
u_{\varepsilon}\stackrel{\ast}{\rightharpoonup} u
                                          &\quad& \mbox{in}\ L^{\infty}([0, \infty); L^2_{\sigma}(\Omega)),\\
u_{\varepsilon}\rightharpoonup u          &\quad& \mbox{in}\ L_{loc}^{\frac{10}{3}}(\bar{\Omega}\times[0, \infty)),\\
\nabla u_{\varepsilon}\rightharpoonup \nabla u
                                          &\quad& \mbox{in}\ L_{loc}^{2}(\bar{\Omega}\times[0, \infty))
\end{eqnarray*}
as $\varepsilon\rightarrow0$. Moreover, using interpolation inequality for $L^p$-norm, we have
\begin{equation*}
n_{\varepsilon}\rightarrow n            \quad \mbox{in}\ L_{loc}^{\frac{10}{7}}(\bar{\Omega}\times[0, \infty))
\end{equation*}
as $\varepsilon\rightarrow0$. According to (\ref{est. c L infite}) and the Lebesgue dominated convergence theorem, we obtain
\begin{eqnarray*}
F'_{\varepsilon}(n_{\varepsilon})\chi(c_{\varepsilon})c_{\varepsilon}^{\frac{3}{4}}\rightarrow \chi(c)c^{\frac{3}{4}}
                                                         &\quad& \mbox{in}\ L_{loc}^{20}(\bar{\Omega}\times[0, \infty)),\\
f(c_{\varepsilon})\rightarrow f(c)                       &\quad& \mbox{in}\ L_{loc}^{\frac{10}{3}}(\bar{\Omega}\times[0, \infty)),\\
c_{\varepsilon}\rightarrow c                             &\quad& \mbox{in}\ L_{loc}^{2}(\bar{\Omega}\times[0, \infty)),\\
F_{\varepsilon}(n_{\varepsilon})\rightarrow n            &\quad& \mbox{in}\ L_{loc}^{\frac{10}{7}}(\bar{\Omega}\times[0, \infty)),\\
Y_{\varepsilon}u_{\varepsilon}\rightarrow u              &\quad& \mbox{in}\ L_{loc}^{2}(\bar{\Omega}\times[0, \infty))
\end{eqnarray*}
as $\varepsilon\rightarrow0$. Therefore,
\begin{eqnarray*}
n_{\varepsilon}u_{\varepsilon}\rightharpoonup nu         &\quad& \mbox{in}\ L_{loc}^{1}(\bar{\Omega}\times[0, \infty)),\\
n_{\varepsilon}F'_{\varepsilon}(n_{\varepsilon})\chi(c_{\varepsilon})\nabla c_{\varepsilon}\rightharpoonup n\chi(c)\nabla c
                                                         &\quad& \mbox{in}\ L_{loc}^{1}(\bar{\Omega}\times[0, \infty)),\\
F_{\varepsilon}(n_{\varepsilon})f(c_{\varepsilon})\rightarrow nf(c)
                                                         &\quad& \mbox{in}\ L_{loc}^{1}(\bar{\Omega}\times[0, \infty)),\\
c_{\varepsilon}u_{\varepsilon}\rightarrow cu             &\quad& \mbox{in}\ L_{loc}^{1}(\bar{\Omega}\times[0, \infty)),\\
Y_{\varepsilon}u_{\varepsilon}\otimes u\rightarrow u\otimes u
                                                         &\quad& \mbox{in}\ L_{loc}^{1}(\bar{\Omega}\times[0, \infty))
\end{eqnarray*}
as $\varepsilon\rightarrow0$.
Based on the above convergence properties, we can pass to the limit in each term of weak formulation for (\ref{RegularizedCNS}) to construct a global weak solution of (\ref{CNS})-(\ref{initial c}) and thereby completes the proof.

%\noindent\emph{Acknowledgements}

%%%%%%%%%%%%%%%%%%%%%%%%%%%%%%%%%%%%%%%%%%%%%%%%%%%%%%%%%%%%%%%%%%%%%%%%%%%%%%%%

\end{document}